\documentclass[12pt]{article}
\usepackage{amsfonts, amssymb, epsfig,subfigure}
\usepackage{mathrsfs}
\usepackage{stmaryrd}
\usepackage{graphicx,amsmath}
\usepackage{latexsym}
\usepackage{verbatim}
\usepackage{diagbox}
\usepackage{color}
\newtheorem{theorem}{Theorem}[section]
\newtheorem{lemma}{Lemma}[section]

\newtheorem{expl}{Example}[section]
\newtheorem{remark}{Remark}[section]

\newtheorem{thm}{Theorem}[section]

\newtheorem{rem}[thm]{Remark}

\makeatletter
\@addtoreset{equation}{section}

\newcommand{\beq}[1]{\begin{equation} \label{#1}}
\newcommand{\eeq}{\end{equation}}
\newcommand{\bed}{\begin{displaymath}}
\newcommand{\eed}{\end{displaymath}}
\newcommand{\bea}{\bed\begin{array}{rl}}
\newcommand{\eea}{\end{array}\eed}

\newcommand{\barray}{\begin{array}{ll}}
\newcommand{\earray}{\end{array}}
\newcommand{\lf}{\lfloor}

\newcommand{\Spvek}[2][r]{%
  \gdef\@VORNE{1}
  \left(\hskip-\arraycolsep%
    \begin{array}{#1}\vekSp@lten{#2}\end{array}%
  \hskip-\arraycolsep\right)}

\def\vekSp@lten#1{\xvekSp@lten#1;vekL@stLine;}
\def\vekL@stLine{vekL@stLine}
\def\xvekSp@lten#1;{\def\temp{#1}%
  \ifx\temp\vekL@stLine
  \else
    \ifnum\@VORNE=1\gdef\@VORNE{0}
    \else\@arraycr\fi%
    #1%
    \expandafter\xvekSp@lten
  \fi}
\allowdisplaybreaks[2]
\makeatletter
  \newcommand\figcaption{\def\@captype{figure}\caption}
  \newcommand\tabcaption{\def\@captype{table}\caption}
\makeatother

\def\sqr#1#2{{\vcenter{\vbox{\hrule height.#2pt
\hbox{\vrule width.#2pt height#1pt \kern#1pt
\vrule width.#2pt} \hrule height.#2pt}}}}

\newcommand{\E}{\mathbb{E}}
\newcommand{\PP}{\mathbb{P}}
\newcommand{\RR}{\mathbb{R}}

\def\br{\breve}

\def\a{\alpha} 
\def\e{\varepsilon}   
 \def\k{\kappa}

\def\tr{\triangle} \def\lf{\left} \def\rt{\right}
\def\b{\beta}

\def\<{\langle} \def\>{\rangle}

\def\1{\oslash} \def\2{\oplus} \def\3{\otimes} \def\4{\ominus}
\def\5{\circ} \def\6{\odot} \def\7{\backslash} \def\8{\infty}
\def\9{\bigcap} \def\0{\bigcup} \def\+{\pm} \def\-{\mp}
\def\[{\langle} \def\]{\rangle}

\newcommand{\dis}{\displaystyle}

\def\nn{\nonumber}

\def\bc{\begin{center}}       \def\ec{\end{center}}
\def\ba{\begin{array}}        \def\ea{\end{array}}
\def\be{\begin{equation}}     \def\ee{\end{equation}}
\def\bea{\begin{eqnarray}}    \def\eea{\end{eqnarray}}
\def\beaa{\begin{eqnarray*}}  \def\eeaa{\end{eqnarray*}}
\def\la{\label}
\begin{document}
\title{An explicit
approximation for super-linear stochastic functional differential equations}

\author{Xiaoyue Li
\thanks{School of Mathematics and Statistics,
Northeast Normal University, Changchun,  130024, China. Research
of this author  was supported by  the National Natural Science Foundation of China (No. 11971096), the National Key R$\&$D Program of China (2020YFA0714102), the Natural Science Foundation of Jilin Province (No. YDZJ202101ZYTS154), the Education Department of Jilin Province (No. JJKH20211272KJ), and the Fundamental Research Funds for the Central Universities.}
\and Xuerong Mao
\thanks{Department of Mathematics and Statistics, University of Strathclyde, Glasgow G1 1XH, UK. Research of this author was supported by
the Royal Society (WM160014, Royal Society Wolfson Research Merit Award),
the Royal Society of Edinburgh (RSE1832),
and Shanghai Administration of Foreign Experts Affairs (21WZ2503700, the Foreign Expert Program).
}
\and
Guoting Song\thanks{School of Mathematics and Statistics,
Northeast Normal University, Changchun, 130024, China. 
}
\date{}}
\maketitle

\begin{abstract}
Since it is difficult to  implement implicit schemes  on the infinite-dimensional space,
we aim to develop the explicit numerical method for approximating super-linear stochastic functional differential equations (SFDEs).
Precisely, borrowing the truncation idea and linear interpolation we propose an explicit  truncated Euler-Maruyama (EM) scheme for SFDEs,
and obtain the  boundedness and  convergence in $L^p$ $(p\geq2)$. We also yield the convergence rate with $1/2$ order.
Different from some previous works
\cite{mao2003, zhang-song-liu}, we release the global Lipschitz restriction on the diffusion
coefficient.
Furthermore, we reveal that  numerical solutions preserve the underlying exponential stability.
Moreover, we give several examples to support our theory.
\end{abstract}

 \vspace{3mm}
\section{Introduction}\label{Tntr}
In recent years, stochastic functional differential equations (SFDEs) have been used in a wide spectrum of application, including  economics, mechanics, biological sciences, and control theory among others; see \cite{BCDDR, KMCW, Kushner2008, Mohammed, NNY2021I, NNY2021II,SM, stoica}. The theory of SFDEs has been investigated extensively; see \cite{Cont-Fournie, dupire2009, LMS2011, Mao2007, Mohammed, Dang2020,Shaikhet, vRM}.
Since it is almost impossible to get the explicit solutions of SFDEs,  numerical solutions become the viable alternative.
For stochastic delay differential equations (SDDEs), which form a special and important
class of SFDEs, a good summary and detailed discussion
on the numerical analysis can be found, e.g., in
\cite{Bao2013, Buckwar-Kuske-Mohammed-Shardlow, cao2015, DKS2016, Hu-Moha-Yan2004, mao2011, Wang-Gan-Wang, ZH}.
 In this paper we  focus on the numerical method of  nonlinear SFDE
\begin{align}\la{o1.1}
\left\{
\begin{array}{ll}
\mathrm{d}x(t)=f(x_t)\mathrm{d}t +g(x_t)\mathrm{d}B(t),~~~t>0,&\\
~~x(t)=\xi(t),~~~t\in[-\tau,~0],&\\
\end{array}
\right.
\end{align}
 where $\tau>0$ is a constant,  $f: C([-\tau,~0];\RR^n) \rightarrow \RR^{n}$ and $g: C([-\tau,~0];\RR^n)  \rightarrow \RR^{n\times d}$ are local Lipschitz continuous,
 $B(\cdot)$ is a $d$-dimensional
Brownian motion on a filtered probability space $( \Omega,~\mathcal{F},~\{{\mathcal{F}}_{t}\}_{t\geq0}, ~\PP )$
satisfying the usual conditions (it is increasing and right
continuous while ${\mathcal{F}}_{0}$ contains all $\mathbb{P}$-null sets),
the initial data $\xi(\cdot)\in C([-\tau,~0];\RR^n)$,
and $x_t=\{x_t(\theta): \theta\in [-\tau, 0]\}$ is $C([-\tau,0];\RR^n)$-valued stochastic process on $t\geq 0$, in which $x_t(\theta)=x(t+\theta)$.
Our primary objectives are to construct an explicit numerical algorithm  for  nonlinear SFDE \eqref{o1.1} and further establish the finite-time strong convergence and  long-time  stability theory.

Numerical algorithms are well investigated for SFDEs with globally Lipschitz continuous coefficients, such as Euler-Maruyama (EM) method \cite{mao2003}, $\theta$-EM method \cite{Buckwar2004},
semi-implicit EM method \cite{Ding-Wu-Liu2006},  split-step backward EM method \cite{Jiang-Shen-Hu}. More often than not, SFDEs arising in practice are  highly nonlinear.
 ~As a result, some implicit numerical procedures have been developed to combat the nonlinearity, such as the backward EM method \cite{zhou-jin}, the split-step theta method for stochastic delay integro differential equations (SDIDEs) \cite{zhao-yi-xu2020}.
Compared with  stochastic differential equations (SDEs), it is more difficult to implement  implicit schemes for nonlinear SFDEs on the infinite-dimensional space.
To tackle this difficulty,
the  backward EM scheme \cite{zhou-jin} is  implicit only with respect to the current state $x(t)$,
 and \cite{zhao-yi-xu2020} required  that the integrand $h(\cdot, \cdot)$ in the integro term $\int_{t-\tau}^{t}h(s,x(s))\mathrm ds$ is  globally Lipschitz.
Although the strong  convergence in \cite{zhao-yi-xu2020, zhou-jin} have been studied, their convergence rate  are still open.
Owing to the simple algebraic structures and cheap computational cost,
a truncated EM scheme with order less than 1/2 was proposed \cite{zhang-song-liu} for nonlinear SFDEs with one-sided Lipschitz drift coefficient and  global Lipschitz diffusion coefficient.
It is necessary to point out that above methods are not applicable for a large  class of SFDEs with the super-linear diffusion term, 
such as the stochastic volatility model
\begin{align}\la{e1}
\left\{
\begin{array}{ll}
\mathrm dx(t)=\big(1+4x(t)-4x^{3}(t)\big)\mathrm dt
+2\int_{-1/2}^{0}x^{2}(t+\theta)\mathrm d\theta\mathrm dB(t),~~~t>0,\\
x(t)=t-1,~~~t\in[-1/2,0].
\end{array}
\right.
\end{align}
This SFDE   may be regarded as  the functional version of the stochastic volatility model introduced by Hutzenthaler and Jentzen \cite[p.89, (4.63)]{HJ2014} and its more general versions have been discussed in \cite{Chang-Youree, SM}.
Naturally, under less restrictive conditions, to construct an efficient  explicit numerical scheme and to establish the strong convergence theory for  nonlinear SFDE \eqref{o1.1} become our top priority.

Moreover, the stability analysis of the numerical methods attracts more and more attentions.
 Although stability of numerical solutions for SDIDEs has been investigated widely \cite{Ding-Wu-Liu2006, Jiang-Shen-Hu, li-gan2012, liu-deng2020, RB2008, zhao-yi-xu2020}, less research has been done
for SFDEs. To our best knowledge, there are only a few papers so far. For example,
Wu, Mao and Kloeden \cite{Wu-Mao-Kloeden2011} examined the almost surely  exponential stability of the EM method
for linear SFDEs.
Zhou, Xie and Fang \cite{ZXF2014}  utilized the backward EM scheme to approximate
the almost surely exponential stability  of nonlinear SFDEs, and
Liu and Deng \cite{liu-deng2021} went a further step to  study
this using  the complete backward EM scheme.
As far as we know, it is still open how to design an explicit scheme to approximate the stability of super-linear SFDEs.
Therefore, the other aim of this paper is to look for an explicit numerical method to capture the exponential stability of super-linear SFDEs.

Borrowing the truncation idea from \cite{Li2018} and linear interpolation,
we propose an explicit truncated EM scheme for nonlinear SFDEs, and obtain the strong convergence and stability theory.
Our main contributions are highlighted as follows:
\begin{itemize}
\item Under flexible Khasminskii-type conditions ((A1) and (A2)),   the explicit truncated EM scheme for super-linear SFDEs is proposed. It is proved that numerical solutions are bounded  and converge to the exact solutions in the $p$th moment  ($p\geq2$).
\item The optimal rate of convergence, $1/2$ order,  is also yielded under additional polynomial growth  coefficient conditions.
Compared with the previous works
\cite{mao2003, zhang-song-liu}, the global Lipschitz restriction on the diffusion coefficient is released.
\item It is examined that the explicit numerical solutions preserve the long-time asymptotic
properties of the exact ones, including the  exponential stability in $L^p$ and with probability $1$ ($\PP-1$).
\end{itemize}

The rest of the paper is organized as follows.
Section \ref{NP} introduces some notations and preliminary  results on  the exact solutions of SFDE \eqref{o1.1}.
Section \ref{MR} proposes an explicit truncated EM scheme for SFDE \eqref{o1.1} and
studies the strong convergence and convergence rate.
Section \ref{ES} goes a further step
to investigate the exponential stability of numerical solutions in $L^p$ and $\PP-1$.
Section \ref{Nexamp} provides several examples to illustrate our results.

\section{Preliminary results}\label{NP}
Throughout this paper,
the following notations are used.
Let $\mathbb{N}$ denote  the set of the non-negative integers.
Let $|\cdot|$ denote both the Euclidean norm in $\RR^n$ and the trace  norm in $\RR^{n\times d}$.
Let $\lfloor a \rfloor$ denote the integer part of the real
number $a$. Let $a\vee{}b=\max\{a,b\}$ and $a\wedge{}b=\min\{a,b\}$ for real numbers $a,~b$.
Let $A^T$ denote the transpose of the vector or matric $A$.
Let $\boldsymbol{1}_A(\cdot)$ denote the indicator function of the set $A$.
Define $x/|x|=\textbf 0$ when $|x|=0$.
Denote by $ \mathcal C:=C([-\tau,~0];~\mathbb{R}^{n})$ the space of all continuous functions
 $\phi(\cdot)$ from
$[-\tau,~0]$ to $\RR^n$ equipped with the  norm $\|\phi\|
=\sup_{-\tau\leq\theta\leq0}|\phi(\theta)|$. Let $\RR_+=[0,\infty)$ and
denote by
   $C^{2} (\RR^n;  ~{{\RR}}_+)$ the space of all continuously twice
 differentiable  nonnegative functions
  defined on $\RR^n$.
For any  $V \in C^{2} (\RR^n;   ~\RR_+)$, define an
operator ${\cal{L}}V$
 $:  \mathcal C\to  \RR$ by
\begin{align*}
{\cal{L}}V(\phi) =&V_x(\phi(0))f(\phi)
 +\frac{1}{2}\hbox{~trace} (g^T(\phi)V_{xx}(\phi(0))g(\phi)),
\end{align*}
where
$$V_x(x)=\left(\frac{\partial V(x)}{\partial x_1},\cdots,\frac{\partial V(x)}{\partial x_n}\right),
~~~~V_{xx}(x)=\left(\frac{\partial^2V(x)}{\partial x_i\partial x_j}\right)_{n\times n}.$$
Denote by $ \mathcal{W}:=\mathcal{W}([-\tau,~0];~\mathbb{R}_+)$ the space of all bounded,  Borel-measurable  functions $\rho(\cdot)$ from $[-\tau,~0]$ to $\RR_+$ such that $\int_{-\tau}^{0}\rho(\theta)\mathrm d\theta=1$.
In addition, $L$ denotes  a generic positive constant, independent of $m$, and $\tr$ (used below), which may take different values at  different appearances.

For convenience we impose the following hypotheses.

\textbf{(A1).} For each  $\ell>0$, there exists a  $R_{\ell}>0$ such that
\begin{align*}
  |f(\phi)-f(\bar \phi)|\vee |g(\phi)-g(\bar \phi)|
  \leq{}
   R_\ell\Big(|\phi(0)-\bar \phi(0)|+\frac{1}{\tau}\int_{-\tau}^{0}|\phi(\theta)-\bar \phi(\theta)|\mathrm d\theta\Big)
  \end{align*}
for all $\phi,~\bar \phi\in \mathcal C$ with $\|\phi\|\vee\|\bar \phi\|\leq \ell$.

\textbf{(A2).} There exist positive constants $p,~\varrho,~ a_1,~a_2,~a_3$ satisfying $p\geq2, ~a_2>a_3$, and  a function  $\rho_1(\cdot)\in\mathcal{W}$ such that for any $\phi\in \mathcal C$,
\begin{align}\la{assA2}
&2(\phi(0))^Tf(\phi)+(p-1)|g(\phi)|^2\nn\\
\leq &a_1\big(1+|\phi(0)|^{2}+\frac{1}{\tau}\int_{-\tau}^{0}|\phi(\theta)|^{2}\mathrm d\theta\big)-a_2|\phi(0)|^{2+\varrho}+a_3\int_{-\tau}^{0}|\phi(\theta)|^{2+\varrho}\rho_1(\theta)\mathrm d\theta.
\end{align}
%
\begin{rem}\la{ass12}
If there exists  a function $\tilde\rho(\cdot)\in\mathcal W$ such that
\begin{align*}
&2(\phi(0))^Tf(\phi)+(p-1)|g(\phi)|^2\nn\\
\leq &a_1\big(1+|\phi(0)|^{2}+\int_{-\tau}^{0}|\phi(\theta)|^{2}\tilde \rho(\theta)\mathrm d\theta\big)-a_2|\phi(0)|^{2+\varrho}+a_3\int_{-\tau}^{0}|\phi(\theta)|^{2+\varrho}\rho_1(\theta)\mathrm d\theta,
\end{align*}
due to
$$\int_{-\tau}^{0}|\phi(\theta)|^2\tilde\rho(\theta)\mathrm d\theta\leq L \int_{-\tau}^{0}|\phi(\theta)|^2\mathrm d\theta,$$
 \eqref{assA2} follows by choosing $a_1:=a_1(\tau L\vee1)$.
%

\end{rem}


Now we prepare some  results on the exact solution  to end this section.
\begin{theorem}\la{oth1}
 Assume that
  $(\textup{A}1)$ and $(\textup{A}2)$ hold. Then SFDE \eqref{o1.1} with the initial data $\xi\in \mathcal C$  has a unique global solution $x(\cdot)$ on $[-\tau,\infty)$  satisfying
   \be\la{o2.7}\sup_{-\tau\leq t<\infty}
   \E|x(t)|^p \leq L.
   \ee
For any constant $m>\|\xi\|$ and  $T>0$,
\begin{align} \la{o2.9}
\mathbb{P}\{\delta_{m}\leq T\}\leq\frac{L_1}{m^p},
\end{align}
where
 \begin{align}\label{o2.8}
\delta_{m}:=\inf\left\{ t\geq 0:|x(t)|\geq m\right\},
 \end{align}
 and the constant $L_1$ depends on $T$ but  is independent of $m$.
\end{theorem}
\begin{proof}\textbf {Proof.}
Under $(\textup{A}1)$ and $(\textup{A}2)$,  in view of \cite[Theorem 2]{LMS2011}, SFDE \eqref{o1.1} with the initial data $\xi\in \mathcal C$ admits a unique global solution.
Moreover, for any $\e>0$ and $\phi\in \mathcal C$,  using \eqref{assA2} and the Young inequality yields
\begin{align*}
&2(\phi(0))^Tf(\phi)+(p-1)|g(\phi)|^2\nn\\
\leq&\Big(a_1+(a_1\e^{-\frac{2}{2+\varrho}})(\e^{\frac{2}{2+\varrho}}|\phi(0)|^2)
+(a_1\e^{-\frac{2}{2+\varrho}})\frac{\e^{\frac{2}{2+\varrho}}}{\tau}\int_{-\tau}^{0}|\phi(\theta)|^{2}\mathrm d\theta\Big)-a_2|\phi(0)|^{2+\varrho}\nn\\
&+a_3\int_{-\tau}^{0}|\phi(\theta)|^{2+\varrho}\rho_1(\theta)\mathrm d\theta\nn\\
\leq
&(a_1+2a_1^{\frac{2+\varrho}{\varrho}}\e^{-\frac{2}{\varrho}})
-(a_2-\e)|\phi(0)|^{2+\varrho}+a_3\int_{-\tau}^{0}|\phi(\theta)|^{2+\varrho}\rho_1(\theta)\mathrm d\theta+\e(\frac{1}{\tau}\int_{-\tau}^{0}|\phi(\theta)|^{2}\mathrm d\theta)^{\frac{2+\varrho}{2}}.
\end{align*}
Letting $\e=(a_2-a_3)/4$ and applying the H\"{o}lder inequality leads to
\begin{align*}
&2(\phi(0))^Tf(\phi)+(p-1)|g(\phi)|^2\nn\\
\leq &L-(a_2-\frac{a_2-a_3}{4})|\phi(0)|^{2+\varrho}
+a_3\int_{-\tau}^{0}|\phi(\theta)|^{2+\varrho}\rho_1(\theta)\mathrm d\theta+\frac{a_2-a_3}{4}(\frac{1}{\tau}\int_{-\tau}^{0}|\phi(\theta)|^{2+\varrho}\mathrm d\theta).
\end{align*}
Define $V(x)=|x|^{p}$.
Applying the  Young inequality and  the H\"{o}lder inequality we derive
\begin{align}\la{o2.10+1}
&{\cal{L}}V(\phi)
\leq\frac{p}{2}|\phi(0)|^{p-2}
\big(2(\phi(0))^Tf(\phi)+(p-1)|g(\phi)|^2\big)\nn\\
\leq&L-\frac{p}{2}\big(a_2-\frac{a_2-a_3}{2}-\frac{(p-2) }{p+\varrho}(a_3+\frac{a_2-a_3}{4})\big)|\phi(0)|^{p+\varrho}\nn\\
&+\frac{p(2+\varrho) a_3}{2(p+\varrho)}\int_{-\tau}^{0}|\phi(\theta)|^{p+\varrho}\rho_1(\theta)\mathrm d\theta
+\frac{p(2+\varrho) (a_2-a_3)}{8(p+\varrho)}(\frac{1}{\tau}\int_{-\tau}^{0}|\phi(\theta)|^{p+\varrho}\mathrm d\theta).
\end{align}
Note that
\begin{align*}
\frac{p(2+\varrho) a_3}{2(p+\varrho)}+\frac{p(2+\varrho) (a_2-a_3)}{8(p+\varrho)}=&\frac{p}{2}\big(a_3-\frac{(p-2)a_3}{p+\varrho}
+\frac{a_2-a_3}{4}-\frac{(p-2)(a_2-a_3)}{4(p+\varrho)}\big)\nn\\
<&\frac{p}{2}\big(a_2-\frac{a_2-a_3}{2}-\frac{(p-2) }{p+\varrho}(a_3+\frac{a_2-a_3}{4})\big).
\end{align*}
Then by similar arguments in the proof of \cite[Theorem 3]{LMS2011}, the required assertion \eqref{o2.7} follows.
Similarly, by virtue of the Dynkin  formula  and \eqref{o2.10+1} we derive that for any $T>0$,
\begin{align*}
\E|x(T\wedge \delta_m)|^p= &|\xi(0)|^p+\E\int^{T\wedge \delta_m}_{0}{\cal{L}}V(x_t)\mathrm dt\leq\|\xi\|^p+LT\leq L_1,
\end{align*}
where the constant $L_1$ depends on $T$ but  is independent of $m$.
Therefore,
\begin{align*}
m^p\PP\{\delta_{m}\leq T\}=&\E(|x(\delta_{m})|^p\textbf 1_{\{\delta_{m}\leq T\}})\leq
\E|x(T\wedge \delta_m)|^p\leq L_1,
\end{align*}
which implies the desired assertion \eqref{o2.9}.
\end{proof}
\section{Numerical scheme and convergence in $L^p$}\label{MR}
The purpose of this section is to construct an  explicit scheme for  SFDE \eqref{o1.1} and further to establish the strong convergence.
Under (A1), we  choose a strictly increasing continuous function $\Gamma(\cdot)$ from $[1, +\infty)$ to $\RR_+$ such that $\lim_{l\rightarrow+\infty}\Gamma(l)=+\infty$ and
 \begin{align}\label{o3.1}
 \sup_{\substack{\|\phi\|\vee \|\bar \phi\|\leq l\\
 \phi\not=\bar\phi}} \dis\Big(\frac{ |f (\phi)-f (\bar\phi)|}{(\Psi(\phi,\bar\phi))^{1/2}}\vee\frac{| g(\phi)-g(\bar\phi)|^2 }{\Psi(\phi,\bar\phi)}\Big)
 \leq~\Gamma(l),~~l\geq1,
 \end{align}
where $\phi,~\bar\phi\in \mathcal C$, and
 $$\Psi(\phi,\bar\phi):=|\phi(0)-\bar\phi(0)|^2+\frac{1}{\tau}\int_{-\tau}^{0}|\phi(\theta)-\bar\phi(\theta)|^2\mathrm d\theta.$$
 Denote by $\Gamma^{-1}(\cdot$) the inverse function of $\Gamma(\cdot)$ from $[\Gamma(1), +\infty)$ to $\RR_{+}$.
For a stepsize $\tr\in(0,1]$ which is a fraction of $\tau$, say $\Delta=\tau/N$ for some integer $N \ge \tau$, define a truncation mapping $\Lambda^{\triangle,\lambda}_{\Gamma}:\RR^n\rightarrow \RR^n$ by
 \begin{align}\la{tru}
\Lambda^{\triangle,\lambda}_{\Gamma}(x)= \Big(|x|\wedge \Gamma^{-1}(K\tr^{-\lambda})\Big) \frac{x}{|x|},
\end{align}
where $K:=\Gamma(1)\vee|f(\textbf 0)|\vee|g(\textbf 0)|^2$,  with a slight  abuses of notation $\textbf 0$ denotes the zero element in ${\cal C}$, namely the function $\phi(\theta)\equiv 0$ for $ \theta\in [-\tau, 0]$,  and $\lambda\in(0,1/2)$.
Let $t_{k}=k\tr$ for any $k=-N,-N+1\cdots$, and define the truncated EM scheme by
 \begin{align}\la{o3.3}
\left\{
\begin{array}{llllll}
\br Y^{\tr}(t_k)=\xi(t_k), ~~~k=-N, -N+1,\cdots,0,&\\
Y^{\tr}(t_{k})=\Lambda^{\triangle,\lambda}_{\Gamma}(\breve Y^{\tr}(t_{k})),~~~k=-N,-N+1,\ldots,&\\
\breve Y^{\tr}(t_{k+1})=Y^{\tr}(t_k)+f(\bar Y_{t_k}^{\tr})\tr +g(\bar Y_{t_k}^{\tr})\triangle B_k,~~~k=0,1,\ldots,&
\end{array}
\right.
\end{align}
where $\triangle B_k=B(t_{k+1})-B(t_{k})$, $\bar Y_{t_k}^{\tr}=\{\bar Y_{t_k}^{\tr}(\theta):-\tau\leq \theta\leq0\}$ is a $\mathcal C$-valued random variable defined by
\begin{align}\label{o3.4}
\bar Y_{t_k}^{\tr}(\theta)&=Y^{\tr}(t_{k+j})+\frac{\theta-j\tr}{\tr}\big(Y^{\tr}(t_{k+j+1})-Y^{\tr}(t_{k+j})\big)\nn\\
&=\frac{(j+1)\tr-\theta}{\tr}Y^{\tr}(t_{k+j})+\frac{\theta-j\tr}{\tr}Y^{\tr}(t_{k+j+1})
\end{align}
for $j\tr\leq \theta\leq(j+1)\tr,~~j=-N,~-N+1,~\cdots,-1$.
\subsection{Uniform moment boundedness of numerical solutions}
We begin with proving that the numerical solutions preserve the uniform  boundedness of exact ones in $L^p$.
Under (A2) we can fix a constant $\e_0\in(0,1)$ such that
\begin{align}\la{e0}
\frac{a_2+a_3}{2}<a_2(1-\e_0)^{\frac{p}{2}}.
\end{align}
For any $\phi\in \mathcal C$, define
\begin{align}\la{varphi}
\Phi(\phi)=1+(1-\e_0)|\phi(0)|^2+\frac{\e_0}{\tau}\int_{-\tau}^{0}|\phi(\theta)|^2\mathrm d\theta.
\end{align}
By \eqref{o3.1}--\eqref{o3.4} we derive that for any $k\in\mathbb N$,
 \begin{align*}
 \big|f(\bar Y_{t_k}^{\tr})\big|
\leq& \Gamma\big(\Gamma^{-1}(K\tr^{-\lambda})\big)\big(|Y^{\tr}(t_k)|^2+\frac{1}{\tau}\int_{-\tau}^{0}|\bar Y^{\tr}_{t_k}(\theta)|^2\mathrm d\theta\big)^{1/2}+|f(\textbf 0)|\nn\\
\leq&2K\tr^{-\lambda}\big(1+|Y^{\tr}(t_k)|^2+\frac{1}{\tau}\int_{-\tau}^{0}|\bar Y^{\tr}_{t_k}(\theta)|^2\mathrm d\theta\big)^{1/2}\nn\\
\leq&2K\tr^{-\lambda}((1-\e_0)\wedge\e_0)^{-1/2}\big(1+(1-\e_0)|Y^{\tr}(t_k)|^2+\frac{\e_0}{\tau}\int_{-\tau}^{0}|\bar Y^{\tr}_{t_k}(\theta)|^2\mathrm d\theta\big)^{1/2},
\end{align*}
which implies
 \begin{align}\la{o3.6}
(\Phi(\bar Y_{t_k}^{\tr}))^{-1}\big|f(\bar Y_{t_k}^{\tr})\big|^2
\leq & L\tr^{-2\lambda}.
\end{align}
Similarly,
\begin{align}\la{o3.6+}
 (\Phi(\bar Y_{t_k}^{\tr}))^{-1}\big|g(\bar Y_{t_k}^{\tr})\big|^2
\leq L\tr^{-\lambda}.
 \end{align}
Due to the definition of $\Phi$,  for the desired assertion
 $\sup_{\tr\in(0, 1]}\sup_{k\geq-N}
  \E|Y^{\tr}(t_k)|^p\leq L$, it is sufficient to prove
\begin{align}\label{o3.10**}\sup_{\tr\in(0, 1]}\sup_{k\geq-N}
  \E(\Phi(\bar Y_{t_k}^{\tr}))^{\frac{p}{2}}\leq L .
\end{align}
In fact, it follows from  \eqref{o3.3} that
\begin{align}\label{o3.10}
|\br Y^{\tr}(t_{k+1})|^2=|Y^{\tr}(t_{k})+f(\bar Y^{\tr}_{t_k})\tr+ g(\bar Y^{\tr}_{t_k})\tr B_{k}|^2
=|Y^{\tr}(t_{k})|^2+\varphi_{k},
\end{align}
where
\begin{align*}
\varphi_{k}:=&|f(\bar Y^{\tr}_{t_k})|^2\tr^2+|g(\bar Y^{\tr}_{t_k}) \tr B_{k}|^2+2 (Y^{\tr}(t_{k}))^T f(\bar Y^{\tr}_{t_k})\tr\nn\\
&+2(Y^{\tr}(t_{k}))^T  g(\bar Y^{\tr}_{t_k})\tr B_{k}
+2 f^T(\bar Y^{\tr}_{t_k})g(\bar Y^{\tr}_{t_k}) \tr B_{k}\tr.
\end{align*}
By virtue of  \eqref{o3.3} and \eqref{o3.4}, we derive from  the convex property of $u(x)=x^2$,
\begin{align}\label{o3.11}
&\frac{1}{\tau}\int_{-\tau}^{0}| \bar Y^{\tr}_{t_{k+1}}(\theta)|^2\mathrm {d}\theta
=\frac{1}{\tau}\int_{-\tau}^{0}|\bar  Y^{\tr}_{t_k}(\theta)|^2\mathrm {d}\theta+\frac{1}{\tau}\int_{-\tr}^{0}| \bar Y^{\tr}_{t_{k+1}}(\theta)|^2\mathrm d\theta-\frac{1}{\tau}\int_{-\tau}^{-\tau+\tr}|\bar  Y^{\tr}_{t_k}(\theta)|^2\mathrm d\theta\nn\\
\leq &\frac{1}{\tau}\int_{-\tau}^{0}| \bar Y^{\tr}_{t_k}(\theta)|^2\mathrm {d}\theta+\frac{1}{\tau}\int_{-\tr}^{0}|\frac{-\theta}{\tr} Y^{\tr}(t_{k})+\frac{\tr+\theta}{\tr} Y^{\tr}(t_{k+1})|^2\mathrm d\theta\nn\\
\leq&\frac{1}{\tau}\int_{-\tau}^{0}|  \bar Y^{\tr}_{t_k}(\theta)|^2\mathrm {d}\theta+{  \frac{\tr}{2\tau}| Y^{\tr}(t_{k})|^2+\frac{\tr}{2\tau}|  Y^{\tr}(t_{k+1})|^2} \nn\\
\leq&\frac{1}{\tau}\int_{-\tau}^{0}|\bar  Y^{\tr}_{t_k}(\theta)|^2\mathrm {d}\theta+\frac{\tr}{2\tau}| Y^{\tr}(t_{k})|^2+\frac{\tr}{2\tau}| Y^{\tr}(t_{k})+ f(\bar Y^{\tr}_{t_k})\tr+  g(\bar Y^{\tr}_{t_k})\tr B_{k}|^2\nn\\
\leq &\frac{1}{\tau}\int_{-\tau}^{0}| \bar Y^{\tr}_{t_k}(\theta)|^2\mathrm {d}\theta+\psi_{k},
\end{align}
where
\begin{align*}
\psi_{k}:={  \frac{2\tr}{\tau}| Y^{\tr}(t_{k})|^2}+\frac{3\tr^3}{2\tau}| f(\bar Y^{\tr}_{t_k})|^2+ \frac{3\tr}{2\tau}| g(\bar Y^{\tr}_{t_k}) \tr B_{k}|^2.
\end{align*}
Hence,  inserting (\ref{o3.10}) and \eqref{o3.11} into \eqref{varphi} yields
\begin{align}\label{o3.12}
\Phi(\bar Y_{t_{k+1}}^{\tr})
\leq&1+(1-\e_0)|\br Y^{\tr}(t_{k+1})|^2+\frac{\e_0}{\tau}\int_{-\tau}^{0}| \bar Y^{\tr}_{t_{k+1}}(\theta)|^2\mathrm {d}\theta\nn\\
\leq&\Phi(\bar Y_{t_k}^{\tr})
+(1-\e_0)\varphi_{k}+\e_0\psi_{k}
=\Phi(\bar Y_{t_k}^{\tr})(1+\Theta_{k}),
\end{align}
where
\begin{align*}
\Theta_{k}:=\big(\Phi(\bar Y_{t_k}^{\tr})\big)^{-1}\big((1-\e_0)\varphi_{k}+\e_0\psi_{k}\big),
\end{align*}
and it is easy to see that $\Theta_k > -1$ a.s.
Making use of \cite[Lemma 3.3]{yang2018} and  (\ref{o3.12}) we derive
\begin{align}\la{o3.13}
 \E \Big(\big(\Phi(\bar Y_{t_{k+1}}^{\tr})\big)^{\frac{p}{2}}\big|\mathcal{F}_{t_{k}}\Big)
   \leq& \big(\Phi(\bar Y_{t_k}^{\tr})\big)^{\frac{p}{2}}\Big(1 + \frac{p}{2} \E\big(\Theta_{k}\big|\mathcal{F}_{t_{k}}\big)+ \frac{p(p-2)}{8} \E\big(\Theta_{k}^2\big|\mathcal{F}_{t_{k}}\big)\nn\\
    &+  \E\big(\Theta_{k}^3P_{\ell}(\Theta_{k})\big|\mathcal{F}_{t_{k}}\big)\Big),
\end{align}
where  $\ell\in \mathbb N$ satisfying $2\ell<p\leq 2(\ell+1)$,
 and $P_{\ell}(\cdot)$ is an $\ell$th-order polynomial.
In order to estimate $ \E(\Phi(\bar Y_{t_{k+1}}^{\tr}))^{\frac{p}{2}}$, we prepare several elementary inequalities on $\Theta_{k}$.

\begin{lemma}\la{esti}
Assume that $(\textup{A}1)$ holds.
Then for any  $k\in\mathbb N$,
\begin{align}\la{o3.22}
\E\big(\Theta_{k}\big|\mathcal{F}_{t_{k}}\big)
\leq(1-\e_0) \big(\Phi(\bar Y_{t_k}^{\tr})\big)^{-1}\big(2(Y^{\tr}(t_{k}))^T f(\bar Y^{\tr}_{t_k})+|g(\bar Y^{\tr}_{t_k})|^2\big)\tr+L\tr,
\end{align}
\begin{align}\la{o3.23}
\E\big(\Theta_{k}^2\big|\mathcal{F}_{t_{k}}\big)
\leq 4(1-\e_0)^2\big(\Phi(\bar Y_{t_k}^{\tr})\big)^{-2}\big| (Y^{\tr}(t_{k}))^Tg(\bar Y^{\tr}_{t_k})\big|^2\tr+L\tr,
\end{align}
\begin{align}\label{o3.24}
\E\big(\Theta_{k}^3P_{\ell}(\Theta_{k})\big|\mathcal{F}_{t_{k}}\big)
\leq L\tr.
\end{align}
\end{lemma}
\begin{proof}\textbf{Proof.}
We use  the similar technique as in the proof of \cite[Theorem 3.1]{Song-Li2021}.
The fact that $B_k$ is  independent of $\mathcal{F}_{t_{k}}$
leads to
\begin{align*}
&\E\big((A \tr B_{k})|\mathcal{F}_{t_{k}}\big)=0,
~~~~\E\big(|A\tr B_{k}|^{2}|\mathcal{F}_{t_{k}}\big)=|A|^2\tr, ~~~A\in\mathbb{R}^{n\times d}.
\end{align*}
This implies
\begin{align*}
&\big(\Phi(\bar Y_{t_k}^{\tr})\big)^{-1}\E\big(\varphi_{k}\big|\mathcal{F}_{t_{k}}\big)\nn\\
=& \big(\Phi(\bar Y_{t_k}^{\tr})\big)^{-1}\Big(2(Y^{\tr}(t_{k}))^T f(\bar Y^{\tr}_{t_k})+|g(\bar Y^{\tr}_{t_k})|^2\Big)\tr+\big(\Phi(\bar Y_{t_k}^{\tr})\big)^{-1}|f(\bar Y^{\tr}_{t_k})|^2\tr^2.
\end{align*}
It follows from  \eqref{o3.6} and $\lambda\in(0,1/2)$ that
\begin{align}\label{o3.16}
\big(\Phi(\bar Y_{t_k}^{\tr})\big)^{-1}\E\big(\varphi_{k}\big|\mathcal{F}_{t_{k}}\big)
=\big(\Phi(\bar Y_{t_k}^{\tr})\big)^{-1}\Big(2(Y^{\tr}(t_{k}))^Tf(\bar Y^{\tr}_{t_k})+|g(\bar Y^{\tr}_{t_k})|^2\Big)\tr+L\tr.
\end{align}
One observes
\begin{align*}
\E\big((A \tr B_{k})^{2i-1}|\mathcal{F}_{t_{k}}\big)=0,~~\E\big(|A\tr B_{k}|^{2i}|\mathcal{F}_{t_{k}}\big)\leq L\tr^{i}, ~~~A\in\mathbb{R}^{1\times d},~i\in\mathbb N  ~\hbox{with}~i\geq2,
\end{align*}
These together with \eqref{o3.6} and \eqref{o3.6+} imply
\begin{align}\label{o3.17}
\big(\Phi(\bar Y_{t_k}^{\tr})\big)^{-2}\E\big(\varphi_{k}^{2}\big|\mathcal{F}_{t_{k}}\big)
\leq4\big(\Phi(\bar Y_{t_k}^{\tr})\big)^{-2}\big|(Y^{\tr}(t_{k}))^Tg(\bar Y^{\tr}_{t_k})\big|^2\tr+L\tr,
\end{align}
\begin{align}\label{o3.18}
\pm\big(\Phi(\bar Y_{t_k}^{\tr})\big)^{-j}\E\big(\varphi_{k}^j\big|\mathcal{F}_{t_{k}}\big)
\leq L\tr,~~~j\in \mathbb N ~\hbox{with}~j\geq3.
\end{align}
For any  positive integer $i$,
\begin{align}\label{o3.21}
&\big(\Phi(\bar Y_{t_k}^{\tr})\big)^{-{i}}\E\big(\psi_{k}^{{i}}\big|\mathcal{F}_{t_{k}}\big)\nn\\
\leq &L\big(\Phi(\bar Y_{t_k}^{\tr})\big)^{-{i}} \Big(\tr^{i}|Y^{\tr}(t_{k})|^{2{i}}
+\tr^{3i}|f(\bar Y^{\tr}_{t_k})|^{2{i}}+ \tr^{2i}|g(\bar Y^{\tr}_{t_k})|^{2{i}} \Big)
\leq L\tr.
\end{align}
Similarly, for any positive integer $i$ and $j$,
\begin{align}\label{o3.20}
\big(\Phi(\bar Y_{t_k}^{\tr})\big)^{-({i}+{j})}\E\big(|\varphi_{k}|^{i}\psi_{k}^{j}
\big|\mathcal{F}_{t_{k}}\big)
\leq L\tr.
\end{align}
Recalling the definition of $\Theta_{k}$, we derive  from \eqref{o3.16} and \eqref{o3.21} that
\begin{align*}
  \E\big(\Theta_{k}\big|\mathcal{F}_{t_{k}}\big)
\leq &(1-\e_0)\big(\Phi(\bar Y_{t_k}^{\tr})\big)^{-1}\big(2(Y^{\tr}(t_{k})^Tf(\bar Y^{\tr}_{t_k})+|g(\bar Y^{\tr}_{t_k})|^2\big)\tr+L\tr.
\end{align*}
Using \eqref{o3.17},  \eqref{o3.21}, \eqref{o3.20}, and the nonnegativity of $\psi_{k}$ yields
\begin{align*}
\E\big(\Theta_{k}^2\big|\mathcal{F}_{t_{k}}\big)
\leq& \big(\Phi(\bar Y_{t_k}^{\tr})\big)^{-2}\E\Big(\big((1-\e_0)^2\varphi_{k}^2
+2|\varphi_{k}|\psi_{k}
+\psi_{k}^2\big)\big|\mathcal{F}_{t_{k}}\Big)\nn\\
\leq& 4(1-\e_0)^2\big(\Phi(\bar Y_{t_k}^{\tr})\big)^{-2}\big|(Y^{\tr}(t_{k}))^Tg(\bar Y^{\tr}_{t_k})\big|^2\tr+L\tr.
\end{align*}
By \eqref{o3.18}--\eqref{o3.20} we obtain that for any integer $j\geq 3$ and  $c\in \RR$,
\begin{align*}
c\E\big( \Theta_{k}^j\big|\mathcal{F}_{t_{k}}\big)
\leq\big(\Phi(\bar Y_{t_k}^{\tr})\big)^{-j}\E\Big(\big( c\varphi_{k}^j+|c|\sum_{i=1}^{j}\frac{j!}{i!(j-i)!}|\varphi_{k}|^{j-i}\psi_{k}^{i}\big|\mathcal{F}_{t_{k}}\Big)
\leq L\tr,
\end{align*}
which implies that \eqref{o3.24} holds.
Therefore the desired results follows.
\end{proof}

To prove the moment boundedness of the numerical solutions we gives the following inequality.
 \begin{lemma}\label{+r}
Assume that \textup{(A2)} holds. Then for any $\phi\in \mathcal C$ and $c>0$,
\begin{align}\la{re1+4}
&c\big(\Phi(\phi)\big)^{\frac{p}{2}}
+\frac{(1-\e_0)p}{2}\big(\Phi(\phi)\big)^{\frac{p-2}{2}}
\big(2(\phi(0))^T f(\phi)+(p-1)|g(\phi)|^2\big)\nn\\
\leq & L-\a_1|\phi(0)|^{p+\varrho}
+\frac{\a_2}{\tau}\int_{-\tau}^{0}|\phi(\theta)|^{p+\varrho}\mathrm d\theta+\a_3\int_{-\tau}^{0}|\phi(\theta)|^{p+\varrho}\rho_1(\theta)\mathrm d\theta,
\end{align}
where   $\alpha_i,~ i=1,2,3, $ are positive constants
  given by \eqref{akappa}, and
satisfy $ \a_2+\a_3< \a_1$.
\end{lemma}
\begin{proof}\textbf{Proof.}
For any $\phi\in \mathcal C$, making use of \eqref{assA2} and \eqref{varphi} leads to
\begin{align*}
&\big(\Phi(\phi)\big)^{\frac{p-2}{2}}
\big(2(\phi(0))^T f(\phi)+(p-1)|g(\phi)|^2\big)\nn\\
\leq&a_1\big(\Phi(\phi)\big)^{\frac{p-2}{2}}\big(1+|\phi(0)|^{2}
+\frac{1}{\tau}\int_{-\tau}^{0}|\phi(\theta)|^{2}\mathrm d\theta\big)
-a_2\big(\Phi(\phi)\big)^{\frac{p-2}{2}}|\phi(0)|^{2+\varrho}\nn\\
&+a_3\big(\Phi(\phi)\big)^{\frac{p-2}{2}}\int_{-\tau}^{0}|\phi(\theta)|^{2+\varrho}\rho_1(\theta)\mathrm d\theta\nn\\
\leq&a_1((1-\e_0)\wedge\e_0)^{-1}\big(\Phi(\phi)\big)^{\frac{p}{2}}
-a_2(1-\e_0)^{\frac{p-2}{2}}|\phi(0)|^{p+\varrho}\nn\\
&+a_3\big(\Phi(\phi)\big)^{\frac{p-2}{2}}\int_{-\tau}^{0}|\phi(\theta)|^{2+\varrho}\rho_1(\theta)\mathrm d\theta.
\end{align*}
Applying the Young inequality yields that for any $c>0$,
\begin{align}\la{re1+2}
&c\big(\Phi(\phi)\big)^{\frac{p}{2}}
+\frac{(1-\e_0)p}{2}\big(\Phi(\phi)\big)^{\frac{p-2}{2}}
\big(2(\phi(0))^T f(\phi)+(p-1)|g(\phi)|^2\big)\nn\\
\leq& L\big(\Phi(\phi)\big)^{\frac{p}{2}}-\frac{pa_2(1-\e_0)^{\frac{p}{2}}}{2}|\phi(0)|^{p+\varrho}
+\frac{pa_3(p-2)}{2(p+\varrho)}\big(\Phi(\phi)\big)^{\frac{p+\varrho}{2}}\nn\\
&+\frac{pa_3(2+\varrho)}{2(p+\varrho)}\big(\int_{-\tau}^{0}|\phi(\theta)|^{2+\varrho}\rho_1(\theta)\mathrm d\theta\big)^{\frac{p+\varrho}{2+\varrho}}\nn\\
\leq& L-\frac{pa_2(1-\e_0)^{\frac{p}{2}}}{2}|\phi(0)|^{p+\varrho}
+\frac{p}{2}(\frac{a_3(p-2)}{p+\varrho}+\frac{a_2-a_3}{2})\big(\Phi(\phi)\big)^{\frac{p+\varrho}{2}}\nn\\
&+\frac{pa_3(2+\varrho)}{2(p+\varrho)}\big(\int_{-\tau}^{0}|\phi(\theta)|^{2+\varrho}\rho_1(\theta)\mathrm d\theta\big)^{\frac{p+\varrho}{2+\varrho}}.
\end{align}
By virtue of \eqref{e0}, choose a  $\kappa\in(0,1)$ sufficiently small such that
\begin{align}\la{kp}
\frac{(a_2+a_3)(1+\kappa)}{2}<a_2(1-\e_0)^{\frac{p}{2}}.
\end{align}
According to \textup{\cite[p.211, Lemma 4.1]{Mao2007}} we know that for any $a,~b\in\RR$, and  $q\geq1$,  there is a constant $\bar L=\bar L(\k)>1$ such that
\begin{align*}
|a+b|^{q}\leq \bar L|a|^{q}+(1+\k)|b|^{q}.
\end{align*}
The above inequality together with the convex property of $u(x)=x^{\frac{p+\varrho}{2}}$ implies
\begin{align}\la{re1+1}
\big(\Phi(\phi)\big)^{\frac{p+\varrho}{2}}
\leq& \bar L
+(1+\kappa)\big((1-\e_0)|\phi(0)|^2
+\frac{\e_0}{\tau}\int_{-\tau}^{0}|\phi(\theta)|^2\mathrm{d}\theta\big)^{\frac{p+\varrho}{2}}\nn\\
\leq&\bar L
+(1+\kappa)\Big((1-\e_0)|\phi(0)|^{p+\varrho}
+\e_0(\frac{1}{\tau}\int_{-\tau}^{0}|\phi(\theta)|^2\mathrm{d}\theta\big)^{\frac{p+\varrho}{2}}\Big).
\end{align}
Inserting \eqref{re1+1} into \eqref{re1+2} and  using the H\"{o}lder inequality we arrive at
\begin{align*}
&c\big(\Phi(\phi)\big)^{\frac{p}{2}}
+\frac{(1-\e_0)p}{2}\big(\Phi(\phi)\big)^{\frac{p-2}{2}}
\big(2(\phi(0))^T f(\phi)+(p-1)|g(\phi)|^2\big)\nn\\
\leq&L
-\frac{p}{2}\big(a_2(1-\e_0)^{\frac{p}{2}}
-(1+\k)(1-\e_0)\Big(\frac{a_3(p-2)}{p+\varrho}+\frac{a_2-a_3}{2} \Big)|\phi(0)|^{p+\varrho}\nn\\
&+\frac{p(1+\k)\e_0}{2}(\frac{a_3(p-2)}{p+\varrho}+\frac{a_2-a_3}{2})
\big(\frac{1}{\tau}\int_{-\tau}^{0}|\phi(\theta)|^{2}\mathrm d\theta\big)^{\frac{p+\varrho}{2}}\nn\\
&+\frac{pa_3(2+\varrho)}{2(p+\varrho)}\Big(\int_{-\tau}^{0}|\phi(\theta)|^{2+\varrho}\rho_1(\theta)\mathrm d\theta\Big)^{\frac{p+\varrho}{2+\varrho}}\nn\\
\leq& L-\a_1|\phi(0)|^{p+\varrho}
+\frac{\a_2}{\tau}\int_{-\tau}^{0}|\phi(\theta)|^{p+\varrho}\mathrm d\theta+\a_3\int_{-\tau}^{0}|\phi(\theta)|^{p+\varrho}\rho_1(\theta)\mathrm d\theta,
\end{align*}
where
\begin{align}\la{akappa}
&\a_1:=\frac{p}{2}\big(a_2(1-\e_0)^{\frac{p}{2}}
-(1+\k)(1-\e_0)(\frac{a_3(p-2)}{p+\varrho}+\frac{a_2-a_3}{2})\big),\nn\\
&\a_2:=\frac{p(1+\k)\e_0}{2}(\frac{a_3(p-2)}{p+\varrho}+\frac{a_2-a_3}{2}),\nn\\
&\a_3:=\frac{a_3p(2+\varrho)}{2(p+\varrho)}.
 \end{align}
It follows from \eqref{kp} that
 $0< \a_2+\a_3< \a_1$.
The proof is therefore complete.
\end{proof}

  \begin{theorem}\la{oth2}
Assume that $(\textup{A}1)$ and $(\textup{A}2)$ hold.  Then,
   \be\la{+o3.9}
  \sup_{\tr\in(0, 1]}\sup_{k\geq-N}
  \E|Y^{\tr}(t_{k})|^{p}\leq L.
   \ee
  \end{theorem}
\begin{proof}\textbf{Proof.}
For any  $\tr\in(0,1]$ and $i\in\mathbb{N}$,
putting (\ref{o3.22})--(\ref{o3.24}) into (\ref{o3.13}) leads to
\begin{align}\la{cnex}
&\E \Big(\big(\Phi(\bar Y_{t_{i+1}}^{\tr})\big)^{\frac{p}{2}}\big|\mathcal{F}_{t_{i}}\Big)\nn\\
\leq&(1+L\tr)\big(\Phi(\bar Y_{t_{i}}^{\tr})\big)^{\frac{p}{2}}+\frac{(1-\e_0)p}{2}\big(\Phi(\bar Y_{t_{i}}^{\tr})\big)^{\frac{p-4}{2}}
\Big(\Phi(\bar Y_{t_i}^{\tr})\big(2(Y^{\tr}(t_{i}))^Tf(\bar Y^{\tr}_{t_i})\nn\\
&+|g(\bar Y^{\tr}_{t_i})|^2\big)+(1-\e_0)(p-2)\big|(Y^{\tr}(t_{i}))^Tg(\bar Y^{\tr}_{t_i})\big|^2\Big)\tr\nn\\
\leq&(1+L\tr)\big(\Phi(\bar Y_{t_{i}}^{\tr})\big)^{\frac{p}{2}}+\frac{(1-\e_0)p}{2}
\big(\Phi(\bar Y_{t_i}^{\tr})\big)^{\frac{p-2}{2}}\big(2(Y^{\tr}(t_{i}))^Tf(\bar Y^{\tr}_{t_i})+(p-1)|g(\bar Y^{\tr}_{t_i})|^2\big)\tr.
\end{align}
 Due to $ \a_2+\a_3< \a_1$ in  Lemma \ref{+r}, we can choose a positive constant $\sigma$ such that
\begin{align*}
e^{\sigma\tau}(\a_2+\a_3)\leq \a_1.
\end{align*}
An application of Lagrange's
mean value theorem yields
$e^{\sigma t_{i+1}}-e^{\sigma t_{i}}\leq e^{\sigma t_{i+1}}\sigma\tr$.
This, along with  \eqref{cnex} implies that
\begin{align*}
&e^{\sigma t_{i+1}}\E \Big(\big(\Phi(\bar Y_{t_{i+1}}^{\tr})\big)^{\frac{p}{2}}\big|\mathcal{F}_{t_{i}}\Big)\nn\\
\leq&(e^{\sigma t_{i}}+e^{\sigma t_{i+1}}\sigma\tr)\big(\Phi(\bar Y_{t_{i}}^{\tr})\big)^{\frac{p}{2}}+e^{\sigma t_{i+1}}L\tr\big(\Phi(\bar Y_{t_{i}}^{\tr})\big)^{\frac{p}{2}}
+\frac{(1-\e_0)pe^{\sigma t_{i+1}}}{2}\big(\Phi(\bar Y_{t_i}^{\tr})\big)^{\frac{p-2}{2}}\nn\\
&\times\big(2(Y^{\tr}(t_{i}))^T f(\bar Y^{\tr}_{t_i})+(p-1)|g(\bar Y^{\tr}_{t_i})|^2\big)\tr\nn\\
\leq&e^{\sigma t_{i}}\big(\Phi(\bar Y_{t_{i}}^{\tr})\big)^{\frac{p}{2}} +e^{\sigma t_{i+1}}\tr\Big(L\big(\Phi(\bar Y_{t_i}^{\tr})\big)^{\frac{p}{2}}+\frac{(1-\e_0)p}{2}
\big(\Phi(\bar Y_{t_i}^{\tr})\big)^{\frac{p-2}{2}}\big(2(Y^{\tr}(t_{i}))^Tf(\bar Y^{\tr}_{t_i})\nn\\
&+(p-1)|g(\bar Y^{\tr}_{t_i})|^2\big)\Big).
\end{align*}
Using \eqref{re1+4} we arrive at
\begin{align*}
&e^{\sigma t_{i+1}}\E \Big(\big(\Phi(\bar Y_{t_{i+1}}^{\tr})\big)^{\frac{p}{2}}\big|\mathcal{F}_{t_{i}}\Big)\nn\\
\leq&e^{\sigma t_{i}}\big(\Phi(\bar Y_{t_{i}}^{\tr})\big)^{\frac{p}{2}} + e^{\sigma t_{i+1}}\tr\Big(L-\a_1| Y^{\tr}(t_i)|^{p+\varrho}+\frac{\a_2}{\tau}\int_{-\tau}^{0}|\bar Y_{t_{i}}^{\tr}(\theta)|^{p+\varrho}\mathrm d\theta\nn\\
&+\a_3\int_{-\tau}^{0}|\bar Y_{t_{i}}^{\tr}(\theta)|^{p+\varrho}\rho_1(\theta)\mathrm d\theta\Big).
\end{align*}
%
For any  $k\in\mathbb N$,
taking expectations in the above inequality
and summing  from $i=0$ to $k$ derives
\begin{align}\label{3.22o}
&e^{\sigma t_{k+1}}\E\big(\Phi(\bar Y_{t_{k+1}}^{\tr})\big)^{\frac{p}{2}}\nn\\
\leq&\big(\Phi(\bar Y_{0}^{\tr})\big)^{\frac{p}{2}}+L\tr\sum_{i=0}^{k}e^{\sigma t_{i+1}}
+\tr\sum_{i=0}^{k}e^{\sigma t_{i+1}}\E\Big(-\a_1| Y^{\tr}(t_i)|^{p+\varrho}\nn\\
&+\frac{\a_2}{\tau}\int_{-\tau}^{0}|\bar Y_{t_{i}}^{\tr}(\theta)|^{p+\varrho}\mathrm d\theta+\a_3\int_{-\tau}^{0}|\bar Y_{t_{i}}^{\tr}(\theta)|^{p+\varrho}\rho_1(\theta)\mathrm d\theta\Big).
\end{align}
It follows from \eqref{o3.4} and the convex property of $u(x)=x^{p+\varrho}$ that
\begin{align*}
&\int_{-\tau}^{0}|\bar Y^{\tr}_{t_i}(\theta)|^{ p+\varrho}\mathrm d\theta\nn\\
=&\sum_{j=-N}^{-1}\int_{j\tr}^{(j+1)\tr}\Big|\frac{(j+1)\tr-\theta}{\tr}Y^{\tr}(t_{i+j})
+\frac{\theta-j\tr}{\tr}Y^{\tr}(t_{i+j+1})\Big|^{p+\varrho}\mathrm d\theta\nn\\
\leq&\sum_{j=-N}^{-1}\int_{j\tr}^{(j+1)\tr}\Big(\frac{(j+1)\tr-\theta}{\tr}
|Y^{\tr}(t_{i+j})|^{p+\varrho}
+\frac{\theta-j\tr}{\tr}|Y^{\tr}(t_{i+j+1})|^{p+\varrho}\Big)\mathrm d\theta\nn\\
=&\frac{\tr}{2}\sum_{j=-N}^{-1}\big(|Y^{\tr}(t_{i+j})|^{p+\varrho}
+|Y^{\tr}(t_{i+j+1})|^{p+\varrho}\big).
\end{align*}
The fundamental theory of calculus  implies
\begin{align}\la{o3.26}
& \frac{1}{\tau}\sum_{i=0}^{k}e^{\sigma t_{i+1}}\int_{-\tau}^{0}|\bar Y^{\tr}_{t_i}(\theta)|^{ p+\varrho}\mathrm d\theta\nn\\
\leq&\frac{e^{\sigma\tau}\tr}{2\tau}\sum_{j=-N}^{-1}\sum_{i=0}^{k}\Big(e^{\sigma t_{i+j+1}}|Y^{\tr}(t_{i+j})
|^{p+\varrho}+e^{\sigma t_{i+j+2}}|Y^{\tr}(t_{i+j+1})|^{p+\varrho}\Big)\nn\\
\leq&\frac{e^{\sigma\tau}\tr}{ \tau}\sum_{j=-N}^{-1} \sum_{v=-N}^{k}e^{\sigma t_{v+1}}|Y^{\tr}(t_{v})
|^{p+\varrho} \nn\\
\leq&e^{\sigma\tau}N\sup_{-N\leq i\leq-1}|\xi(t_i)|^{p+\varrho}+e^{\sigma\tau}\sum_{i=0}^{k}e^{\sigma t_{i+1}}|Y^{\tr}(t_i)|^{p+\varrho}.
\end{align}
Similarly,
\begin{align}\la{o3.26+1}
&\sum_{i=0}^{k}e^{\sigma t_{i+1}}\int_{-\tau}^{0}|\bar Y^{\tr}_{t_i}(\theta)|^{ p+\varrho}\rho_1(\theta)\mathrm d\theta\nn\\
\leq&e^{\sigma\tau}\sum_{j=-N}^{-1}\int_{j\tr}^{(j+1)\tr}\frac{(j+1)\tr-\theta}{\tr}\rho_1(\theta)\mathrm d\theta\sum_{i=0}^{k}e^{\sigma t_{i+j+1}}|Y^{\tr}(t_{i+j})|^{p+\varrho}\nn\\
&+e^{\sigma\tau}\sum_{j=-N}^{-1}\int_{j\tr}^{(j+1)\tr}\frac{\theta-j\tr}{\tr}\rho_1(\theta)\mathrm d\theta \sum_{i=0}^{k}e^{\sigma t_{i+j+2}}|Y^{\tr}(t_{i+j+1})|^{p+\varrho}\nn\\
\leq&e^{\sigma\tau}\sum_{j=-N}^{-1}\int_{j\tr}^{(j+1)\tr}\frac{(j+1)\tr-\theta}{\tr}\rho_1(\theta)\mathrm d\theta\sum_{v=-N}^{k}e^{\sigma t_{v+1}}|Y^{\tr}(t_{v})|^{p+\varrho}\nn\\
&+e^{\sigma\tau}\sum_{j=-N}^{-1}\int_{j\tr}^{(j+1)\tr}\frac{\theta-j\tr}{\tr}\rho_1(\theta)\mathrm d\theta \sum_{v^{'}=-N}^{k}e^{\sigma t_{v^{'}+1}}|Y^{\tr}(t_{v^{'}})|^{p+\varrho}\nn\\
\leq&e^{\sigma\tau}N\sup_{-N\leq i\leq-1}|\xi(t_i)|^{p+\varrho}+e^{\sigma\tau}\sum_{i=0}^{k}e^{\sigma t_{i+1}}|Y^{\tr}(t_i)|^{p+\varrho}.
\end{align}
Inserting \eqref{o3.26} and \eqref{o3.26+1}  into \eqref{3.22o} we arrive at
\begin{align*}
&e^{\sigma t_{k+1}}\E\big(\Phi(\bar Y_{t_{k+1}}^{\tr})\big)^{\frac{p}{2}}\nn\\
\leq&\big(\Phi(\bar Y_{0}^{\tr})\big)^{\frac{p+\varrho}{2}}+L\tr\sum_{i=0}^{k}e^{\sigma t_{i+1}}+(\a_2+\a_3)e^{\sigma\tau}\tau\sup_{-N\leq i\leq-1}|\xi(t_i)|^{p+\varrho}\nn\\
&-(\a_1-e^{\sigma\tau}(\a_2+\a_3))
\tr\sum_{i=0}^{k}e^{\sigma t_{i+1}}\E|Y^{\tr}(t_{i})|^{p+\varrho}\nn\\
\leq& L+L\tr\frac{e^{\sigma t_{k+2}}-e^{\sigma\tr}}{e^{\sigma\tr}-1}\leq L+L\frac{e^{\sigma t_{k+2}}}{\sigma},
\end{align*}
which implies that  \eqref{+o3.9} holds.
The proof is therefore complete.
\end{proof}
\subsection{Convergence in $L^p$}\la{SR}
 In this subsection we will investigate  the  strong convergence of the truncated EM scheme \eqref{o3.3}.
For convenience, define an auxiliary process $Z^{\tr}(\cdot)$ by
\begin{align}\label{o3.27}
\left\{
\begin{array}{ll}
Z^{\tr}(t)=\!Y^{\tr}(t_k)+f(\bar Y^{\tr}_{t_k})(t-t_k )+
 g(\bar Y^{\tr}_{t_k})(B(t)-B(t_k) ), ~t\in [t_k, t_{k+1}),~k\in\mathbb N,&\\
Z^{\tr}(t)=\frac{t_{k+1}-t}{\tr}Y^{\tr}(t_k)+\frac{t-t_{k}}{\tr}Y^{\tr}(t_{k+1}),~ ~t\in [t_k, t_{k+1}],~k=-N,\cdots, -1.
\end{array}
\right.
\end{align}
Obviously,
\begin{align*}
\lim_{t\uparrow t_{k}}Z^{\tr}(t)=&\br Y^{\tr}(t_k),~~~k=1,2,\cdots,\nn\\
Z^{\tr}(t_k)=&Y^{\tr}(t_k),~~~k=-N,-N+1,\cdots.
\end{align*}
For any constant $m>\|\xi\|$
and  $\tr\in (0,1]$, define
\begin{align}\la{o3.29}
& \gamma_{m, \tr}=\inf\big\{t\geq0: |Z^{\tr}(t)|>m\big\}.
\end{align}
Choose a $\tr_{1}\in(0,1]$ sufficiently small such that
$$\Gamma^{-1}(K\tr_1^{-\lambda})\geq m,$$
which implies that for any $\tr\in(0,\tr_1]$, $Z^{\tr}(\cdot)$ is continuous on $[-\tau, \gamma_{m,\tr}]$.
Let $ \eta_{m,\tr}:=\lfloor\gamma_{m,\tr}/\tr\rfloor$.
For any $t\in[t_k,t_{k+1})$ with $k\in \mathbb N$, one has
$$\lfloor (t\wedge\gamma_{m,\tr})/\tr\rfloor=k\wedge \eta_{m,\tr}.$$
Define
\begin{align}\la{segment}
\bar Y^{\tr}_{t}=\sum_{k= 0}^{\infty}\bar Y^{\tr}_{t_k}\textbf 1_{[t_k, t_{k+1})}(t),~~~\forall ~ t\geq0.
\end{align}
Hence it follows from \eqref{o3.27}--\eqref{segment} that for any $\tr\in(0,\tr_1]$,
{
\begin{align}\la{o3.30}
Z^{\tr}(t\wedge\gamma_{m,\tr})
=&Y^{\tr}(t_{k\wedge\eta_{m,\tr}})+f(\bar Y^{\tr}_{t_{k\wedge\eta_{m,\tr}}})(t\wedge\gamma_{m,\tr}-t_{k\wedge\eta_{m,\tr}} )\nn\\
&+g(\bar Y^{\tr}_{t_{k\wedge\eta_{m,\tr}}})(B(t\wedge\gamma_{m,\tr})-B(t_{k\wedge\eta_{m,\tr}}) )\nn\\
=&Y^{\tr}(0)+\int_{0}^{t\wedge \gamma_{m,\tr}}f(\bar Y^{\tr}_{s})\mathrm ds
+\int_{0}^{t\wedge \gamma_{m,\tr}}g(\bar Y^{\tr}_{s})\mathrm dB(s).
\end{align}}

Now we start with estimating the probability  that  $Z^{\tr}(t)$ remains within the bounded set $\{x\in\RR^n: |x|\leq m \}$ for $t\in[-\tau, T]$.
\begin{lemma}\la{oth3}
Assume that  $(\textup{A}1)$ and $(\textup{A}2)$ hold.
Then for any  $\tr\in(0,\tr_1]$ and $T>0$,
\begin{align} \la{o3.31}
\mathbb{P}\{\gamma_{m,\tr}\leq T\}\leq\frac{L_2}{m^p},
\end{align}
where the constant $L_2$ depends on $T$ but is independent of $m$ and $\tr$.
  \end{lemma}
\begin{proof}
\textbf{Proof.}
We use the similar techniques
in proofs of  Lemma \ref{esti} and Theorem \ref{oth2}.
Let $\tr\in(0,\tr_1]$ and  $T>0$. We simply write $\gamma=\gamma_{m,\tr},~\eta=\eta_{m,\tr}$.
For any integer $i\in\mathbb N$, if $\eta\geq i+1$,  it follows from \eqref{o3.3}  that
\begin{align*}
Y^{\tr}(t_{(i+1)\wedge \eta})
=&\br Y^{\tr}(t_{(i+1)\wedge \eta})=\br Y^{\tr}(t_{i+1})\nn\\
=&Y^{\tr}(t_{i})+f(\bar Y^{\tr}_{t_i})\tr+ g(\bar Y^{\tr}_{t_i})\tr B_i.
\end{align*}
If $ \eta< i+1$, it is obvious that $ \eta\leq i$. One has
\begin{align*}
&Y^{\tr}(t_{(i+1)\wedge \eta})=Y^{\tr}(t_{ \eta})
=Y^{\tr}(t_{i\wedge \eta}).
\end{align*}
The above two cases implies that for any $i\in\mathbb N$,
\begin{align*}
|Y^{\tr}(t_{(i+1)\wedge \eta})|^2
=&\big|Y^{\tr}(t_{i\wedge \eta})+\big(f(\bar Y^{\tr}_{t_i})\tr+ g(\bar Y^{\tr}_{t_i})\tr B_i\big)\textbf 1_{\{ \eta\geq i+1\}}\big|^2\nn\\
=&|Y^{\tr}(t_{i\wedge \eta})|^2+\varphi_{i}\textbf 1_{\{ \eta\geq i+1\}},
\end{align*}
where $\varphi_i$ is given by the equation below \eqref{o3.10}. In a similar way as \eqref{o3.11} was shown, we obtain
\begin{align*}
\frac{1}{\tau}\int_{-\tau}^{0}| \bar Y^{\tr}_{t_{(i+1)\wedge \eta}}(\theta)|^2\mathrm {d}\theta
=&\frac{1}{\tau}\int_{-\tau}^{0}| \bar Y^{\tr}_{t_{i\wedge \eta}}(\theta)|^2\mathrm {d}\theta
+\frac{1}{\tau}\int_{-\tr}^{0}| \bar Y^{\tr}_{t_{i+1}}(\theta)|^2\mathrm {d}\theta\textbf 1_{\{ \eta\geq i+1\}}\nn\\
&-\frac{1}{\tau}\int_{-\tau}^{-\tau+\tr}|\bar Y^{\tr}_{t_i}(\theta)|^2\mathrm {d}\theta\textbf 1_{\{ \eta\geq i+1\}}\nn\\
\leq &\frac{1}{\tau}\int_{-\tau}^{0}|\bar Y^{\tr}_{t_{i\wedge \eta}}(\theta)|^2\mathrm {d}\theta+\psi_{i}\textbf 1_{\{\eta\geq i+1\}},
\end{align*}
where $\psi_i$ is given by the equation below \eqref{o3.11}. Those implies
\begin{align*}
\Phi(\bar Y_{t_{(i+1)\wedge \eta}}^{\tr})
\leq&1+(1-\e_0)|\br Y^{\tr}(t_{(i+1)\wedge \eta})|^2+\frac{\e_0}{\tau}\int_{-\tau}^{0}| \bar Y^{\tr}_{t_{(i+1)\wedge \eta}}(\theta)|^2\mathrm {d}\theta\nn\\
\leq&\Phi(\bar Y_{t_{i\wedge \eta}}^{\tr})
+\big((1-\e_0)\varphi_{i}+\e_0\psi_{i}\big)\textbf 1_{\{ \eta\geq i+1\}}\nn\\
=&\Phi(\bar Y_{t_{i\wedge \eta}}^{\tr})(1+\Theta_{i}\textbf 1_{\{ \eta\geq i+1\}}),
\end{align*}
where $\Theta_i$ is given by the equation below \eqref{o3.12}. In a similar way as \eqref{o3.13} was shown, we can get
\begin{align}\la{o3.34}
&\E\Big(\big(\Phi(\bar Y^{\tr}_{t_{(i+1)\wedge \eta}})\big)^{\frac{p}{2}}
\big|\mathcal{F}_{t_{i\wedge \eta}}\Big)\nn\\
\leq &\big(\Phi(\bar Y^{\tr}_{t_{i\wedge \eta}})\big)^{\frac{p}{2}}
\Big(1+\frac{p}{2}\E\big(\Theta_{i}\textbf 1_{\{ \eta\geq i+1\}}|\mathcal{F}_{t_{i\wedge \eta}}\big)+\frac{p(p-2)}{8}\E\big(\Theta_{i}^2\textbf 1_{\{ \eta\geq i+1\}}|\mathcal{F}_{t_{i\wedge \eta}}\big)\nn\\
&+\E\big(\Theta_{i}^3P_{\ell}( \Theta_{i})\textbf 1_{\{ \eta\geq i+1\}}|\mathcal{F}_{t_{i\wedge \eta}}\big)\Big).
\end{align}
Note that $$\tr B_i\textbf 1_{\{ \eta\geq i+1\}}=B(t_{(i+1)\wedge \eta})-B(t_{i\wedge  \eta}).$$
An application of the  Doob martingale stopping
 time theorem \cite[p.11, Theorem 3.3]{Mao2007}  yields
\begin{align}\label{o3.35}
&\E\big(A \tr B_i\textbf 1_{\{ \eta\geq i+1\}}\big|
\mathcal{F}_{t_{i\wedge \eta}}\Big)=0,\nn\\
&\E\Big(\big|A\tr B_i\big|^{2}\textbf 1_{\{ \eta\geq i+1\}}\big|
\mathcal{F}_{t_{i\wedge \eta}}\Big)
=|A|^2\tr\textbf 1_{\{ \eta\geq i+1\}},~~~\forall A\in\RR^{n\times d}.
\end{align}
Similarly,
\begin{align}\la{o3.36}
&\E\Big((A \tr B_i)^{2j-1}\textbf 1_{\{ \eta\geq i+1\}}\big|
\mathcal{F}_{t_{i\wedge \eta}}\Big)=0,\nn\\
&\E\Big(\big|A\tr B_i\big|^{2j}\textbf 1_{\{ \eta\geq i+1\}}\big|
\mathcal{F}_{t_{i\wedge \eta}}\Big)
\leq L\tr^{j}\textbf 1_{\{ \eta\geq i+1\}},
\end{align}
where $A\in\mathbb{R}^{1\times d}$, $j\in\mathbb N$  with $j\geq 2$.
 Using the similar technique as in the proof of Lemma \ref{esti}, it follows from \eqref{o3.6}, \eqref{o3.6+}, \eqref{o3.35}, and \eqref{o3.36} that
\begin{align}\la{tilde3.22}
&\E\big(\Theta_{i}\textbf 1_{\{ \eta\geq i+1\}}\big|\mathcal{F}_{t_{i\wedge  \eta}}\big)\nn\\
\leq&\Big((1-\e_0)\big(\Phi(\bar Y_{t_i})\big)^{-1}\big(2(Y^{\tr}(t_{i}))^Tf(\bar Y^{\tr}_{t_i})+|g(\bar Y^{\tr}_{t_i})|^2\big)\tr+L\tr\Big)\textbf 1_{\{ \eta\geq i+1\}},
\end{align}
\begin{align}\la{tilde3.23}
&\E\big(\Theta_{i}^2\textbf 1_{\{ \eta\geq i+1\}}\big|\mathcal{F}_{t_{i\wedge  \eta}}\big)\nn\\
\leq&\Big(4(1-\e_0)^2\big(\Phi(\bar Y_{t_i})\big)^{-2}\big|(Y^{\tr}(t_{i}))^Tg(\bar Y^{\tr}_{t_i})\big|^2\tr+L\tr\Big)\textbf 1_{\{ \eta\geq i+1\}},
\end{align}
and
\begin{align}\la{tilde3.24}
\E\big(\Theta_{i}^3P_{\ell}( \Theta_{i})\textbf 1_{\{ \eta\geq i+1\}}\big|\mathcal{F}_{t_{i\wedge \eta}}\big)
\leq L\tr\textbf 1_{\{ \eta\geq i+1\}}.
\end{align}
Plugging \eqref{tilde3.22}--\eqref{tilde3.24} back into \eqref{o3.34} and using  \eqref{re1+4}, we derive
\begin{align}\la{tildeY}
&\E\Big(\big(\Phi(\bar Y^{\tr}_{t_{(i+1)\wedge \eta}})\big)^{\frac{p}{2}}
\big|\mathcal{F}_{t_{i\wedge \eta}}\Big)\nn\\
\leq &(1+L\tr)\big(\Phi(\bar Y^{\tr}_{t_{i\wedge \eta}})\big)^{\frac{p}{2}}
+\frac{(1-\e_0)p\tr}{2}
\textbf 1_{\{ \eta\geq i+1\}}
\big(\Phi(\bar Y_{t_i}^{\tr})\big)^{\frac{p-2}{2}}\big(2(Y^{\tr}(t_{i}))^T f(\bar Y^{\tr}_{t_i})\nn\\
&+(p-1)|g(\bar Y^{\tr}_{t_i})|^2\big)\nn\\
\leq&(1+L\tr)\big(\Phi(\bar Y^{\tr}_{t_{i\wedge \eta}})\big)^{\frac{p}{2}}
+\tr
\textbf 1_{\{ \eta\geq i+1\}}\Big(-\a_1
\big( Y^{\tr}(t_i)\big)^{p+\varrho}+\frac{\a_2}{\tau}\int_{-\tau}^{0}|\bar Y^{\tr}_{t_i}(\theta)|^{p+\varrho}\mathrm d\theta\nn\\
&+\a_3
\int_{-\tau}^{0}|\bar Y^{\tr}_{t_i}(\theta)|^{p+\varrho}\rho_1(\theta)\mathrm d\theta\big).
\end{align}
Taking expectations and summing both sides of \eqref{tildeY} from $i=0$ to $k$ we obtain
\begin{align}\la{3.42o}
&\E\big(\Phi(\bar Y_{t_{(k+1)\wedge \eta}}^{\tr})\big)^{\frac{p}{2}}
\leq \big(\Phi(\bar Y_{0}^{\tr}\big)^{\frac{p}{2}}
+L\tr\sum_{i=0}^{k}\E\big(\Phi(\bar Y_{t_{i\wedge \eta}}^{\tr})\big)^{\frac{p}{2}}\nn\\
&+\tr\sum_{i=0}^{k}
\E\Big(\textbf 1_{\{ \eta\geq i+1\}}\big(-\a_1
\big(Y^{\tr}(t_i)\big)^{p+\varrho}+\frac{\a_2}{\tau}\int_{-\tau}^{0}|\bar Y^{\tr}_{t_i}(\theta)|^{p+\varrho}\mathrm d\theta\nn\\
&+\a_3
\int_{-\tau}^{0}|\bar Y^{\tr}_{t_i}(\theta)|^{p+\varrho}\rho_1(\theta)\mathrm d\theta\big)\Big).
\end{align}
One notices that
$
 \textbf 1_{\{ \eta\geq i+1\}}
\leq \textbf 1_{\{ \eta\geq i+j+1\}}
$ for any $j\in\{-N,-N+1,\cdots, 0\}$.
In the similar ways as \eqref{o3.26} and \eqref{o3.26+1} were shown, we derive
\begin{align}\la{3.43o}
&\frac{1}{\tau}\sum_{i=0}^{k}\int_{-\tau}^{0}|\bar Y^{\tr}_{t_i}(\theta)|^{p+\varrho}\mathrm d\theta\textbf 1_{\{ \eta\geq i+1\}}\nn\\
\leq&\frac{\tr}{2\tau}\sum_{j=-N}^{-1}\sum_{i=0}^{k}\Big(|Y^{\tr}(t_{i+j})
|^{p+\varrho}\textbf 1_{\{ \eta\geq i+1\}}+|Y^{\tr}(t_{i+j+1})|^{p+\varrho}\textbf 1_{\{ \eta\geq i+1\}}\Big)\nn\\
\leq&\frac{\tr}{2\tau}\sum_{j=-N}^{-1}\sum_{i=0}^{k}\Big(|Y^{\tr}(t_{i+j})
|^{p+\varrho}\textbf 1_{\{ \eta\geq i+j+1\}}+|Y^{\tr}(t_{i+j+1})|^{p+\varrho}\textbf 1_{\{ \eta\geq i+j+2\}}\Big)\nn\\
\leq&N\sup_{-N\leq i\leq-1}|\xi(t_{i})|^{p+\varrho}+\sum_{i=0}^{k}|Y^{\tr}(t_{i})|^{p+\varrho}
\textbf 1_{\{ \eta\geq i+1\}},
\end{align}
and
\begin{align}\la{3.43o+1}
&\sum_{i=0}^{k}\int_{-\tau}^{0}|\bar Y^{\tr}_{t_i}(\theta)|^{p+\varrho}\rho_1(\theta)\mathrm d\theta\textbf 1_{\{ \eta\geq i+1\}}\nn\\
\leq &N\sup_{-N\leq i\leq-1}|\xi(t_{i})|^{p+\varrho}+\sum_{i=0}^{k}|Y^{\tr}(t_{i})|^{p+\varrho}
\textbf 1_{\{ \eta\geq i+1\}}.
\end{align}
Inserting \eqref{3.43o} and \eqref{3.43o+1}  into \eqref{3.42o} and using $\a_2+\a_3<\a_1$ leads to
\begin{align*}
\E\big(\Phi(\bar Y_{t_{(k+1)\wedge \eta}}^{\tr})\big)^{\frac{p}{2}}
\leq& L+L\tr\sum_{i=0}^{k}\E\big(\Phi(\bar Y_{t_{i\wedge \eta}}^{\tr})\big)^{\frac{p}{2}}.
\end{align*}
Applying the discrete Gronwall inequality \cite[p.56, Theorem 2.5]{mao-yuan} yields  that for any $T>0$,
\begin{align}\la{o3.37}
\sup_{\tr\in(0,\tr_1]}\sup_{0\leq k\tr\leq T}\E\big(\Phi(\bar Y^{\tr}_{t_{k\wedge \eta}})\big)^{\frac{p}{2}}\leq \tilde L,
\end{align}
where the constant $\tilde L$ depends on $T$ but is independent of $m$ and $\tr$.
Furthermore, for any $t\in[0,T]$ there exists a  $k\in\mathbb N$  such that $t\in[t_k,t_{k+1})$. Using the first equality of \eqref{o3.30} and   the Burkholer-Davis-Gundy inequality
 yields
\begin{align*}
&\E\big|Z^{\tr}(t\wedge\gamma)\big|^p\nn\\
\leq& 3^{p-1}\E\big|Y^{\tr}(t_{k\wedge  \eta})\big|^p+3^{p-1}\E\big|f(\bar Y^{\tr}_{t_{k\wedge \eta}})\big|^p\tr^p+3^{p-1}(\frac{p^{p+1}}{2(p-1)^{p-1}})^{\frac{p}{2}}\E\big|g(\bar Y^{\tr}_{t_{k\wedge  \eta}})\big|^p\tr^{\frac{p}{2}}.
\end{align*}
Recalling the definition of $\Phi$ and making use of  \eqref{o3.6}, \eqref{o3.6+}, and \eqref{o3.37}  leads to
\begin{align*}
&\E\big|Z^{\tr}(t\wedge\gamma)\big|^p\nn\\
\leq& L\E\big|Y^{\tr}(t_{k\wedge  \eta})\big|^p+L\E\big(\Phi(\bar Y^{\tr}_{t_{k\wedge \eta}})\big)^{\frac{p}{2}}\tr^{p(1-\lambda)}+L\E\big(\Phi(\bar Y^{\tr}_{t_{k\wedge \eta}})\big)^{\frac{p}{2}}\tr^{\frac{p(1-\lambda)}{2}}\nn\\
\leq&L\E\big|Y^{\tr}(t_{k\wedge  \eta})\big|^p+L\E\big(\Phi(\bar Y^{\tr}_{t_{k\wedge \eta}})\big)^{\frac{p}{2}}
\leq L\sup_{0\leq k\tr\leq T}\E\big(\Phi(\bar Y^{\tr}_{t_{k\wedge \eta}})\big)^{\frac{p}{2}}\leq L_2,
\end{align*}
which implies
\begin{align*}
    \PP \{\gamma\leq T\}
    \leq&\frac{\E\big|Z^{\tr}(T\wedge\gamma)\big|^p}{m^p }\leq  \frac{L_2}{m^p }.
\end{align*}
The proof is therefore complete.
\end{proof}

Let
\begin{align*}
\vartheta_{m,\tr}:=\delta_m \wedge \gamma_{m,\tr},
\end{align*}
where $\delta_m$ and $\gamma_{m,\tr}$ are given by (\ref{o2.8}) and (\ref{o3.29}), respectively.
For any $t\geq-\tau$, define
\begin{align}\label{o3.8}
Y^{\tr}(t)=&\sum_{k=-N}^{\infty }Y^{\tr}(t_k)\textbf {1}_{[t_k,t_{k+1})}(t).
\end{align}
\begin{lemma}\label{oth4}
Assume that $(\textup{A}1)$ and $(\textup{A}2)$ hold.
Then for any $m>\|\xi\|$,  $\tilde p>0$,  and $T>0$,
\begin{align*}
\lim_{\tr\rightarrow 0^{+}}\E\lf(|x(T)-Y^{\tr}(T)|^{\tilde p} \textup{\textbf 1}_{\{\vartheta_{m,\tr}> T\}}\rt)=0.
\end{align*}
\end{lemma}
\begin{proof}\textbf{Proof.}
Define
\begin{align*}
f_{m}(\phi)=f\Big(\lf(\|\phi\|\wedge m \rt) \frac{\phi}{\|\phi\|}\Big),
~~~~g_{m}(\phi)=g\Big(\lf(\|\phi\|\wedge {m} \rt) \frac{\phi}{\|\phi\|}\Big).
\end{align*}
By virtue of (A1) we derive that for any $\phi,~\bar \phi\in \mathcal C$,
  \begin{align*}
  |f_{m}(\phi)-f_{m}(\bar \phi)|\vee  |g_{m}(\phi)-g_{m}(\bar \phi)|\leq R_{m}\Big(|\phi(0)-\bar\phi(0)|+\frac{1}{\tau}\int_{-\tau}^{0}|\phi(\theta)-\bar\phi(\theta)|\mathrm d\theta\Big),
  \end{align*}
which implies
 \begin{align}\label{o3.40}
 |f_{m}(\phi)|\vee  |g_{m}(\phi)|\leq&( 2R_{m}\vee |f(\textbf{0})|\vee|g(\textbf{0})|)(1+\|\phi\|).
 \end{align}
Consider the following  SFDE
\be\la{linear}
d\tilde x(t) =f_{m}(\tilde x_t)dt +g_{m}(\tilde x_t)dB(t)
\ee
with  the initial data $\xi \in \mathcal C$.
  For each $\tr\in(0,\tr_1]$,
we make use of the EM scheme to approximate SFDE \eqref{linear} and denote by
$\{y^{\tr}(k\tr)\}_{k\geq -N}$ the discrete EM solution (see \cite[(2.3)]{mao2003}).
Define the piecewise constant EM process
$\{y^{\tr}(t)\}_{ t\geq -\tau}$ by
$$y^{\tr}(t)=\sum_{k=-N}^{\infty}y^{\tr}(k\tr)\textbf 1_{[k\tr,(k+1)\tr)}(t),$$
and denote by $ \{z^{\tr}(t)\}_{ t\geq-\tau}$ the corresponding continuous EM solution (see \cite[(2.7)]{mao2003}).
Let $T>0$.
In a similar way as \cite[Theorem 2.5]{mao2003} were shown, we derive
that for any $\tilde p>0$,
\begin{align}\label{o3.42}
&\lim_{\tr\rightarrow 0^+}\E\Big( \sup_{0\leq t\leq T} |\tilde x(t)-z^{\tr}(t)|^{\tilde p}\Big)
=0.
\end{align}
For any $k\in\mathbb N$, denote by $\bar y_{k\tr}^{\tr}(\cdot)$ the $\mathcal C$-valued  linear interpolation of $y^{\tr}((k-N)\tr)$, $\cdots$, $y^{\tr}(k\tr)$ (see \cite[(2.4)]{mao2003}).
Applying \cite[Lemma 3.2]{mao2003} yields
\begin{align}\la{boundlinear}
\E\big(\sup_{0\leq k\tr\leq T} \|\bar y^{\tr}_{k\tr}\|^{\tilde p}\big)
\leq&\|\xi\|^{\tilde p} +\E\big(\sup_{0\leq t\leq T} |z^{\tr}(t)|^{\tilde p}\big)\leq \check{L},
\end{align}
where the positive constant $\check L$ depends on $T$ but is independent of $\tr$.
It is straightforward to see that for any $t\in[t_k,t_{k+1})$ with $k\in\mathbb N$,
\begin{align*}
 |z^{\tr}(t)- y^{\tr}(t)|^{\tilde p}
 =&|f(\bar y^{\tr}_{k\tr})(t-t_k)+g( \bar y^{\tr}_{k\tr})(B(t)-B(t_k))|^{\tilde p}\nn\\
\leq&2^{\tilde p}|f(\bar y^{\tr}_{k\tr})|^{\tilde p}\tr^{\tilde p}+2^{\tilde p}|g(\bar y^{\tr}_{k\tr})|^{\tilde p}|B(t)-B(t_k)|^{\tilde p}.
\end{align*}
According to \eqref{o3.40} yields
\begin{align*}
 |z^{\tr}(t)-y^{\tr}(t)|^{\tilde p}
\leq&L\big(1+\|\bar y^{\tr}_{k\tr}\|^{\tilde p}\big)\big(\tr^{\tilde p}+|B(t)-B(t_k)|^{\tilde p}\big).
\end{align*}
For each  $\tilde p >1$,  applying the H\"{o}lder inequality, the Burkholder-Davis-Gundy inequality, and using \eqref{boundlinear}, we get
\begin{align*}
&\lim_{\tr\rightarrow 0^+}\E\Big(\sup_{0\leq t\leq T} |z^{\tr}(t)- y^{\tr}(t)|^{\tilde p}\Big)\nn\\
\leq&L\lim_{\tr\rightarrow 0^+}\E\Big(\sup_{0\leq k\tr\leq T}\sup_{k\tr\leq t<(k+1)\tr}\big(1+\|\bar y^{\tr}_{k\tr}\|^{\tilde p}\big)\big(\tr^{\tilde p}+|B(t)-B(t_k)|^{\tilde p}\big)\Big)\nn\\
\leq&L\lim_{\tr\rightarrow 0^+}\Big(1+\E\big(\sup_{0\leq k\tr\leq T}\| \bar y^{\tr}_{k\tr}\|^{2\tilde p}\big)\Big)^{\frac{1}{2}}\nn\\
&\times\Big(\tr^{2\tilde p}+\E\big(\sup_{0\leq k\tr\leq T}\sup_{k\tr\leq t<(k+1)\tr}|B(t)-B(t_k)|^{2\tilde p}\big)\Big)^{\frac{1}{2}}\nn\\
\leq&L\lim_{\tr\rightarrow 0^+}\Big(\tr^{2\tilde p}+\sum_{k=0}^{\lfloor T/\tr\rfloor}\E\big(\sup_{k\tr\leq t<(k+1)\tr}|B(t)-B(t_k)|^{2\tilde p}\big)\Big)^{\frac{1}{2}}\nn\\
\leq&L\lim_{\tr\rightarrow 0^+}\Big(\tr^{2\tilde p}+(\lfloor T/\tr\rfloor+1)\tr^{\tilde p}\Big)^{\frac{1}{2}}
=0.
\end{align*}
This, along with the Lyapunov inequality implies that for any  $\tilde p>0$,
\begin{align}\label{o3.43}
 \lim_{\tr\rightarrow 0^+}\E\Big( \sup_{0\leq t\leq T} |z^{\tr}(t)-y^{\tr}(t)|^{\tilde p}\Big)=0.
 \end{align}
Utilizing \eqref{o3.42} and \eqref{o3.43} we have that for any  $\tilde p>0$,
\begin{align}\label{o3.44}
\lim_{\tr\rightarrow 0^+}\E\Big(\sup_{0\leq t\leq T} |\tilde x(t)-y^{\tr}(t)|^{\tilde p}\Big)=0.
\end{align}
According to $m>\|\xi\|$ yields
\be\la{o3.45}
x(t\wedge\vartheta_{m,\tr})=\tilde x(t\wedge\vartheta_{m,\tr}),~\forall~ t\geq -\tau, ~~~\hbox{a.s.},
\ee
and
\begin{align}\la{o3.46}
Y^{\tr}(t\wedge \vartheta_{m,\tr})=y^{\tr}(t\wedge \vartheta_{m,\tr}),
~~\forall ~t~\geq-\tau,~\tr\in(0,\tr_1],~~~\hbox{a.s.}
\end{align}
By virtue of  (\ref{o3.44})--(\ref{o3.46}) we derive that for any $\tilde p>0$,
\begin{align*}
&\lim_{\tr\rightarrow 0^+}\E\lf(|x(T)-Y^{\tr}(T)|^{\tilde p} \textbf 1_{\{\vartheta_{m,\tr}> T\}}\rt)\nn\\
 =&\lim_{\tr\rightarrow 0^+}\E\lf(|\tilde x(T\wedge\vartheta_{m,\tr})-y^{\tr}(T\wedge\vartheta_{m,\tr})|^{\tilde p} \textbf 1_{\{\vartheta_{m,\tr}> T\}}\rt)\nn\\
\leq &\lim_{\tr\rightarrow 0^+}\E\left(|\tilde x(T\wedge\vartheta_{m,\tr})-y^{\tr}(T\wedge\vartheta_{m,\tr})|^{\tilde p}\right)
\leq\lim_{\tr\rightarrow 0^+}\E\left(\sup_{0\leq t\leq T}|\tilde x(t)-y^{\tr}(t)|^{\tilde p}\right)=0.
\end{align*}
 The proof is complete.
\end{proof}
\begin{theorem}\la{th4}
Assume that  $(\textup{A}1)$ and $(\textup{A}2)$ hold.  Then for any $T>0$ and  $q\in (0, p)$,
$$\lim_{\tr\rightarrow 0^+} \E |x(T)-Y^{\tr}(T)|^q=0.$$
\end{theorem}
\begin{proof}\textbf{Proof.}
Let $m>\|\xi\|$,   ~$\tr\in(0,\tr_1]$.  For any $T>0$, $q\in(0,p)$, and $\epsilon_1>0$,
applying the Young inequality yields
\begin{align}\la{o3.47}
 \E|x(T)-Y^{\tr}(T)|^q
 =&\E\big(|x(T)-Y^{\tr}(T)|^q \textbf 1_{\{\vartheta_{m,\tr}\leq T\}}\big)+\E\big(|x(T)-Y^{\tr}(T)|^q \textbf 1_{\{\vartheta_{m,\tr}> T\}}\big)\nn\\
\leq& \frac{q \epsilon_1}{p}\E|x(T)-Y^{\tr}(T)|^{p}+\frac{p-q}{p \epsilon_1^{q/(p-q)}}\PP \big(\vartheta_{m,\tr}\leq T\big)\nn\\
&+\E\big(|x(T)-Y^{\tr}(T)|^q \textbf 1_{\{\vartheta_{m,\tr}> T\}}\big).
\end{align}
For any $\epsilon>0$, in view of Theorem  \ref{oth1} and Theorem \ref{oth2}, choose an $\epsilon_1>0 $ sufficiently small  such that
 \begin{align}\la{th3.2+1}
 \frac{q \epsilon_1}{p}\E\big(|x(T)-Y^{\tr}(T)|^{p} \big)\leq\frac{2^{p-1}q \epsilon_1}{p}\E\big(|x(T)|^p+\sup_{\tr\in(0,1]}|Y^{\tr}(T)|^{p} \big)<\frac{\e}{3}.\end{align}
It follows from \eqref{o2.9} and \eqref{o3.31} that
\begin{align*}
\frac{p-q}{p \epsilon_1^{q/(p-q)}}\PP \big(\vartheta_{m,\tr}\leq T\big)
\leq \frac{p-q}{p \epsilon_1^{q/(p-q)}}\Big(\PP\big\{\delta_{m} \leq T\big\}+\PP\big\{\gamma_{m, \tr}\leq T\big\}\Big)
\leq \frac{(L_1+L_2)(p-q)}{m^pp \epsilon_1^{q/(p-q)}}.
\end{align*}
Choose an $m$ sufficiently large  such that
$$
\frac{(L_1+L_2)(p-q)}{m^pp \epsilon_1^{q/(p-q)}}<\frac{\e}{3},
$$
which implies
\begin{align}\la{th3.2+2}
\frac{p-q}{p \epsilon_1^{q/(p-q)}}\PP \big(\vartheta_{m,\tr}\leq T\big)
<\frac{\e}{3}.
\end{align}
Inserting \eqref{th3.2+1} and \eqref{th3.2+2} into  \eqref{o3.47} and using Lemma \ref{oth4} we derive
\begin{align*}
 \lim_{\tr\rightarrow0^+}\E|x(T)-Y^{\tr}(T)|^q=0.
\end{align*}
 The proof is complete.
\end{proof}
\subsection{Convergence rate in $L^p$ }\la{CR}
To study the convergence rate of the truncated EM scheme, we require the following  notations and assumptions.
 By $\mathcal{U}:=\mathcal{U}(\RR^n\times\RR^n;   ~\RR_+)$, denote the space of all continuous  functions $U(\cdot,\cdot)$  from $\RR^n\times\RR^n$ to $\RR_+$ such that for each $\ell>0$, there exists a  $\bar R_{\ell}>0$ for which
\begin{align}\la{calU}
U(x,\bar x)\leq \bar R_{\ell}|x-\bar x|,~~~\forall ~x,~\bar x\in\RR^n ~\hbox{with} ~|x|\vee|\bar x|\leq \ell.
\end{align}

\textbf{(A3).}
There are positive  constants $\mu\geq1/2$ and $a_4$    such that the initial data $\xi$ satisfies
\begin{align*}
|\xi(s_1)-\xi(s_2)|\leq a_4|s_1-s_2|^\mu,~~~\forall~s_1,~s_2\in[-\tau,~0].
\end{align*}

\textbf{(A4).}
There exist positive  constants $\bar q,~\bar p,~a_5$ satisfying $2\leq \bar q<\bar p$,
 and functions $U(\cdot)\in\mathcal{U}$, $\rho_2(\cdot)\in\mathcal W$ such that for any $~\phi,~\bar\phi\in \mathcal C$,
\begin{align*}
&|\phi(0)-\bar\phi(0)|^{\bar q-2}\Big(2\big(\phi(0)-\bar\phi(0)\big)^T \big(f(\phi)-f(\bar \phi)\big) +(\bar p-1)|g(\phi)-g(\bar\phi)|^2\Big)\nn\\
\leq& a_5\big(|\phi(0)-\bar\phi(0)|^{\bar q}
+\!\frac{1}{\tau}\int_{-\tau}^{0}|\phi(\theta)-\bar\phi(\theta)|^{\bar q}\mathrm d \theta\big)\!-U(\phi(0),\bar\phi(0))
+\!\int_{-\tau}^{0}U(\phi(\theta),\bar\phi(\theta))\rho_2(\theta)\mathrm d\theta.
\end{align*}

\textbf{(A5).}
There is a pair of positive constants $a_6$,~$r$  such that for any $\phi,~\bar\phi\in \mathcal C$,
\begin{align*}
 &|f(\phi)-f(\bar\phi)|\leq a_6\Big(|\phi(0)-\bar\phi(0)|+\frac{1}{\tau}\int_{-\tau}^{0}|\phi(\theta)-\bar\phi(\theta)|\mathrm d\theta\Big)\nn\\
 &\times\Big(1+|\phi(0)|^r+|\bar\phi(0)|^r+\frac{1}{\tau}\int_{-\tau}^{0}|\phi(\theta)|^r\mathrm d\theta+\frac{1}{\tau}\int_{-\tau}^{0}|\bar\phi(\theta)|^r\mathrm d\theta\Big),
\end{align*}
and
\begin{align}\label{o3.52}
&|g(\phi)-g(\bar\phi)|^2\leq a_6\Big(|\phi(0)-\bar\phi(0)|^2
+\frac{1}{\tau}\int_{-\tau}^{0}|\phi(\theta)-\bar\phi(\theta)|^2\mathrm d\theta\Big)\nn\\
 &\times\Big(1+|\phi(0)|^r+|\bar\phi(0)|^r+\frac{1}{\tau}\int_{-\tau}^{0}|\phi(\theta)|^r\mathrm d\theta+\frac{1}{\tau}\int_{-\tau}^{0}|\bar\phi(\theta)|^r\mathrm d\theta\Big).
\end{align}
\begin{rem}\label{r2} \rm{
Under (A2), (A4), and (A5),
by virtue of  \eqref{o3.1}  we may take the function $\Gamma(\cdot)$ by:
   \begin{align*}
\Gamma(l)=\sqrt2a_6(1+4l^r), ~~~\forall ~l\geq1.
\end{align*}
Then,
 \begin{align}\label{o3.54}
\Gamma^{-1}(l)=\big(\frac{l}{4\sqrt 2a_6}-\frac{1}{4}\big)^{\frac{1}{r}},~~~~l\geq 5\sqrt2a_6.
\end{align}
  Define
  \begin{align}\la{deterlambda}
  \lambda=\frac{\bar qr}{2(p-\bar q)}.
  \end{align} Obviously, if  $ \bar q< {p}/({r+1}),$  then $ 0<\lambda< {1}/{2}$.
}\end{rem}

%


We  will  estimate the error between the exact solution $\{x(t)\}_{t\geq -\tau}$ and the auxiliary  process $\{Z^{\tr}(t)\}_{t\geq -\tau}$,  and further obtain the convergence rate of the numerical solution $\{Y^{\tr}(t)\}_{t\geq -\tau}$.
For this purpose, we begin with proving that the moment of  $ Z^{\tr}(t) $ is bounded.
\begin{lemma}\la{bounded}
Assume that $(\textup{A}1)$ and $(\textup{A}2)$ hold.
Then
\begin{align*}
\sup_{\tr\in(0,1]}\sup_{t\geq-\tau}\E|Z^{\tr}(t)|^p\leq L.
\end{align*}
\end{lemma}
\begin{proof}\textbf{Proof.}
Let  $\tr\in(0,1]$.
For any $t\in[t_k, t_{k+1})$ with $k\in\mathbb N$, making use of \eqref{o3.27} yields
\begin{align*}
\E|Z^{\tr}(t)|^p\leq 3^{p-1}\big(\E|Y^{\tr}(t_k)|^p+\E|f(\bar Y^{\tr}_{t_k})|^p\tr^p+\E|g(\bar Y^{\tr}_{t_k})|^p\tr^{\frac{p}{2}}\big).
\end{align*}
One notices from \eqref{o3.6} and \eqref{o3.6+} that
\begin{align*}
\E|Z^{\tr}(t)|^p
\leq& L(1+\tr^{p(1-\lambda)}+\tr^{\frac{p(1-\lambda)}{2}})\E\big(1+|Y^{\tr}(t_k)|^p
+\frac{1}{\tau}\int_{-\tau}^{0}|\bar Y_{t_k}^{\tr}(\theta)|^p\mathrm d\theta\big)\nn\\
\leq &L(1+\sup_{\tr\in(0,1]}\sup_{k\geq -\tau}\E|Y^{\tr}(t_k)|^p).
\end{align*}
Therefore, the desired assertion follows from \eqref{+o3.9}.
\end{proof}

For any $t\geq 0$, define
\begin{align*}
Z^{\tr}_{t}(\theta)=Z^{\tr}(t+\theta),~~~\forall ~\theta\in[-\tau,0].
\end{align*}
Choose a $\bar\tr\in(0,1]$ sufficiently small such that
$\Gamma^{-1}(K\bar\tr^{-\lambda})>\|\xi\|$.
\begin{lemma}\la{ztyt}
Assume that $\textup{(A2)}$--$(\textup A5)$   hold with $p\geq(2\vee r)$.
 Then for any   $\tr\in(0,\bar \tr]$,
\begin{align}\la{o3.57}
\sup_{t\geq0}\sup_{\theta\in[-\tau, 0]}\E |Z_t^{\tr}(\theta)-\bar Y_t^{\tr}(\theta)|^{\frac{2p}{r+2}}\leq L\tr^{\frac{p}{r+2}}.
\end{align}
\end{lemma}
\begin{proof}\textbf{Proof.}
Let  $\tr\in(0,\bar \tr]$. For any $t\geq0$ there exists  a $k\in\mathbb N$  such that $t\in[t_k,t_{k+1})$.
For any $\theta\in[-\tau,0)$,  there exists  a $j\in\{-N, \cdots, -1\}$  such that $\theta\in[t_j,t_{j+1})$.
Clearly,
$$t+\theta\in[t_{k+j}, t_{k+j+2}),~~~t_k+\theta\in[t_{k+j},t_{k+j+1}).$$
According to the range of $t+\theta$, we devide the proof into five cases.\\
\underline{Case 1}.  If $t+\theta\in[t_{k+j}, t_{k+j+1})\subset[-\tau,0)$, by virtue of $\tr\in(0,\bar\tr]$ we get
\begin{align*}
&Z_{t}^{\tr}(\theta)-\bar Y_{t}^{\tr}(\theta)
=Z^{\tr}(t+\theta)-\bar Y^{\tr}(t_k+\theta)\nn\\
=&\frac{t_{k+j+1}-(t+\theta)}{\tr}\xi(t_{k+j})+\frac{(t+\theta)-t_{k+j}}{\tr}\xi(t_{k+j+1})\nn\\
&-(\frac{t_{k+j+1}-(t_k+\theta)}{\tr}\xi(t_{k+j})+\frac{(t_k+\theta)-t_{k+j}}{\tr}\xi(t_{k+j+1}))\nn\\
=&\frac{t-t_k}{\tr}(\xi(t_{k+j+1})-\xi(t_{k+j})).
\end{align*}
Making use of (A3),  we derive
\begin{align*}
\E\big|Z_{t}^{\tr}(\theta)-\bar Y_{t}^{\tr}(\theta)\big|^{\frac{2p}{r+2}}
\leq L\tr^\frac{p}{r+2}.
\end{align*}
\underline{Case 2}. If $t+\theta\in[t_{k+j+1}, t_{k+j+2})\subset[-\tau,0)$, we have
\begin{align*}
Z_{t}^{\tr}(\theta)-\bar Y_{t}(\theta)
=&\frac{t_{k+j+2}-(t+\theta)}{\tr}\xi(t_{k+j+1})+\frac{(t+\theta)-t_{k+j+1}}{\tr}\xi(t_{k+j+2})\nn\\
&-(\frac{t_{k+j+1}-(t_k+\theta)}{\tr}\xi(t_{k+j})+\frac{(t_k+\theta)-t_{k+j}}{\tr}\xi(t_{k+j+1}))\nn\\
=&\frac{(t+\theta)-t_{k+j+1}}{\tr}\big(\xi(t_{k+j+2})-\xi(t_{k+j+1})\big)\nn\\
&+\frac{t_{k+j+1}-(t_k+\theta)}{\tr}(\xi(t_{k+j+1})-\xi(t_{k+j})).
\end{align*}
It follows from (A3) that
\begin{align*}
\E\big|Z_{t}^{\tr}(\theta)-\bar Y_{t}^{\tr}(\theta)\big|^{\frac{2p}{r+2}}
\leq L\tr^\frac{p}{r+2}.
\end{align*}
\underline{Case 3}. If $t+\theta\in[t_{k+j}, t_{k+j+1})\subset [0,\infty)$, making use of \eqref{o3.4},
\eqref{o3.27}, and \eqref{segment}, we derive
\begin{align}\la{zyerr}
&|Z_{t}^{\tr}(\theta)-\bar Y_{t}^{\tr}(\theta)|
=|Z^{\tr}(t+\theta)-\bar Y^{\tr}(t_k+\theta)|\nn\\
=&\big|Y^{\tr}(t_{k+j})+f(\bar Y_{{t_{k+j}}}^{\tr})(t+\theta-t_{k+j})+g(\bar Y_{{t_{k+j}}}^{\tr})\big(B(t+\theta)-B(t_{k+j})\big)\nn\\
&-\big(\frac{(j+1)\tr-\theta}{\tr}Y^{\tr}(t_{k+j})+\frac{\theta-j\tr}{\tr}Y^{\tr}(t_{k+j+1})\big)\big|\nn\\
\leq&\big|f(\bar Y_{{t_{k+j}}}^{\tr})(t+\theta-t_{k+j})+g(\bar Y_{{t_{k+j}}}^{\tr})\big(B(t+\theta)-B(t_{k+j})\big)\big|\nn\\
&+\frac{\theta-j\tr}{\tr}\big|Y^{\tr}(t_{k+j+1})-Y^{\tr}(t_{k+j})\big|.
\end{align}
Using the techniques in  the proof of \cite[(7.21)]{Li2018} and together with \eqref{o3.3} we get
\begin{align}\la{left}
|Y^{\tr}(t_{k+j+1})-Y^{\tr}(t_{k+j})|
\leq &|\br Y^{\tr}(t_{k+j+1})-Y^{\tr}(t_{k+j})|\nn\\
\leq&|f(\bar Y_{{t_{k+j}}}^{\tr})\tr+g(\bar Y_{{t_{k+j}}}^{\tr})\tr B_{k+j}|.
\end{align}
Substituting \eqref{left} into \eqref{zyerr} leads to
\begin{align}\la{+zyerr}
\E\big|Z_{t}^{\tr}(\theta)-\bar Y_{t}^{\tr}(\theta)\big|^{\frac{2p}{r+2}}
\leq L\E\big(|f(\bar Y_{{t_{k+j}}}^{\tr})|^{\frac{2p}{r+2}}\tr^{\frac{2p}{r+2}}
+|g(\bar Y_{{t_{k+j}}}^{\tr})|^{\frac{2p}{r+2}}\tr^{\frac{p}{r+2}}\big).
\end{align}
It follows from \eqref{o3.52}  that
\begin{align*}
&|g(\bar Y_{{t_{k+j}}}^{\tr})|^{\frac{2p}{r+2}}
\leq L |g(\bar Y_{{t_{k+j}}}^{\tr})-g(\mathbf{0})|^{\frac{2p}{r+2}}+
L|g(\mathbf{0})|^{\frac{2p}{r+2}}\nn\\
\leq&L(|Y^{\tr}(t_{k+j})|^2+\frac{1}{\tau}\int_{-\tau}^{0}|\bar Y_{{t_{k+j}}}^{\tr}(\theta)|^2\mathrm d\theta)^\frac{p}{r+2}(1+|Y^{\tr}(t_{k+j})|^r+\frac{1}{\tau}\int_{-\tau}^{0}|\bar Y_{{t_{k+j}}}^{\tr}(\theta)|^r\mathrm d\theta)^{\frac{p}{r+2}}+ L.
\end{align*}
Making use of the Young inequality we obtain
\begin{align*}
|g(\bar Y_{{t_{k+j}}}^{\tr})|^{\frac{2p}{r+2}}
\leq&L(|Y^{\tr}(t_{k+j})|^2+\frac{1}{\tau}\int_{-\tau}^{0}|\bar Y_{{t_{k+j}}}^{\tr}(\theta)|^2\mathrm d\theta)^{\frac{p}{2}}\nn\\
&+L(1+|Y^{\tr}(t_{k+j})|^r+\frac{1}{\tau}\int_{-\tau}^{0}|\bar Y_{{t_{k+j}}}^{\tr}(\theta)|^r\mathrm d\theta)^{\frac{p}{r}}+L.
\end{align*}
Using $p\geq(2\vee r)$ and applying the H\"{o}lder inequality yields
\begin{align}\label{o3.53}
|g(\bar Y_{{t_{k+j}}}^{\tr})|^{\frac{2p}{r+2}}
\leq& L(1+|Y^{\tr}(t_{k+j})|^{p}
+\frac{1}{\tau}\int_{-\tau}^{0}|\bar Y_{{t_{k+j}}}^{\tr}(\theta)|^{p}\mathrm d\theta).
\end{align}
Inserting \eqref{o3.53} into \eqref{+zyerr}, and then using  \eqref{o3.6}  and $\lambda\in(0,1/2)$, we have
\begin{align*}
&\E\big|Z_{t}^{\tr}(\theta)-\bar Y_{t}^{\tr}(\theta)\big|^{\frac{2p}{r+2}}\nn\\
 \leq& L\E\big(1+|Y^{\tr}(t_{k+j})|^2+\frac{1}{\tau}\int_{-\tau}^{0}|\bar Y^{\tr}_{t_{k+j}}(\theta)|^2\mathrm d\theta\big)^{\frac{p}{r+2}}\tr^{\frac{2p(1-\lambda)}{r+2}
 }\nn\\
& +L\E(1+|Y^{\tr}(t_{k+j})|^{p}
+\frac{1}{\tau}\int_{-\tau}^{0}|\bar Y_{{t_{k+j}}}^{\tr}(\theta)|^{p}\mathrm d\theta)\tr^{\frac{p}{r+2}}\nn\\
 \leq&L\E(1+|Y^{\tr}(t_{k+j})|^{p}
+\frac{1}{\tau}\int_{-\tau}^{0}|\bar Y_{{t_{k+j}}}^{\tr}(\theta)|^{p}\mathrm d\theta)\tr^{\frac{p}{r+2}}.
\end{align*}
This, along with \eqref{+o3.9} implies that
\begin{align*}
\E\big|Z_{t}^{\tr}(\theta)-\bar Y_{t}^{\tr}(\theta)\big|^{\frac{2p}{r+2}}
\leq L\tr^{\frac{p}{r+2}}.
\end{align*}
\underline{Case 4}. If $t+\theta\in[t_{k+j+1}, t_{k+j+2})=[0,\tr)$, we observe that $t_k+\theta\in[-\tr,0)$ and
\begin{align*}
\E\big|Z_{t}^{\tr}(\theta)-\bar Y_{t}^{\tr}(\theta)\big|^{\frac{2p}{r+2}}
\leq &L\E\big|f(\bar Y^{\tr}_0)(t+\theta)+g(\bar Y^{\tr}_0)B(t+\theta)\big|^{\frac{2p}{r+2}}\nn\\
&+L\big|\xi(0)
-\xi(-\tr)\big|^\frac{2p}{r+2}
\leq L\tr^\frac{p}{r+2}.
 \end{align*}
\underline{Case 5}. If $t+\theta\in[t_{k+j+1}, t_{k+j+2})\subset [\tr,\infty)$, we compute
\begin{align*}
&\big|Z_{t}^{\tr}(\theta)-\bar Y_{t}^{\tr}(\theta)\big|\nn\\
=&\big|Y^{\tr}(t_{k+j+1})+f(\bar Y_{t_{k+j+1}}^{\tr})(t+\theta-t_{k+j+1})
+g(\bar Y_{t_{k+j+1}}^{\tr})\big(B(t+\theta)-B(t_{k+j+1})\big)\nn\\
&-\big(\frac{(j+1)\tr-\theta}{\tr}Y^{\tr}(t_{k+j})
+\frac{\theta-j\tr}{\tr}Y^{\tr}(t_{k+j+1})\big)\big|\nn\\
\leq &L\big|f(\bar Y_{t_{k+j+1}}^{\tr})(t+\theta-t_{k+j+1})+g(\bar Y_{t_{k+j+1}}^{\tr})\big(B(t+\theta)-B(t_{k+j+1})\big|\nn\\
&+L\big|Y^{\tr}(t_{k+j+1})-Y^{\tr}(t_{k+j})\big|.
\end{align*}
In a similar way as Case 3 was shown, we derive
$$\E\big|Z_{t}^{\tr}(\theta)-\bar Y_{t}^{\tr}(\theta)\big|^\frac{2p}{r+2}
\leq L\tr^\frac{p}{r+2}.$$
Combining the above five cases  together, we get
 $$\sup_{t\geq0}\sup_{\theta\in[-\tau,0)}\E |Z_t^{\tr}(\theta)-\bar Y_t^{\tr}(\theta)|^{\frac{2p}{r+2}}\leq L\tr^{\frac{p}{r+2}}.$$
In addition,
\begin{align*}
\E\big|Z_{t}^{\tr}(0)-\bar Y_{t}^{\tr}(0)\big|^{\frac{2p}{r+2}}
=\E\big|f(\bar Y_{{t_{k}}}^{\tr})(t-t_{k})+g(\bar Y_{{t_{k}}}^{\tr})\big(B(t)-B(t_{k})\big)\big|^{\frac{2p}{r+2}}\leq L\tr^{\frac{p}{r+2}}.
\end{align*}
The proof is therefore complete.
\end{proof}

For any  $\tr\in(0,\bar \tr]$, let $\vartheta_{\tr}:=\delta_{\Gamma^{-1}(K\tr^{-\lambda})}\wedge\gamma_{\Gamma^{-1}(K\tr^{-\lambda}),\tr}$.
\begin{lemma}\la{ratetwo}
Assume that $\textup{(A2)}$, $\textup{(A4)}$, and $(\textup A5)$ hold with
$\bar q\in(0, p/(r+1))$. Then for any  $\tr\in(0,\bar\tr]$ and $T>0$,
\begin{align*}
\E\big(|x(T)-Z^{\tr}(T)|^{\bar q}\textup{\textbf 1}_{\{ \vartheta_{\tr}\leq T\}}\big)\leq L_3\tr^{\frac{\bar q}{2}},
\end{align*}
where the constant $L_3$ depends on $T$ but is independent of  $\tr$.
\end{lemma}
\begin{proof}\textbf{Proof.}
Using the Young inequality we derive
\begin{align*}
\E\big(|x(T)-Z^{\tr}(T)|^{\bar q}\textbf 1_{\{ \vartheta_{\tr}\leq T\}}\big)
\leq \frac{\bar q\tr^{\frac{\bar q}{2}}}{p}\E|x(T)-Z^{\tr}(T)|^{p}
+\frac{p-\bar q}{p\tr^{\frac{\bar q^2}{2(p-\bar q)}}}\PP( \vartheta_{\tr}\leq T).
\end{align*}
In view of Theorem \ref{oth1} and Lemma \ref{bounded},  we get
\begin{align*}
&\frac{\bar q\tr^{\frac{\bar q}{2}}}{p}\E|x(T)-Z^{\tr}(T)|^{p}
\leq  L\tr^{\frac{\bar q}{2}
}\big(\E|x(T)|^{p}+\E|Z^{\tr}(T)|^p\big)
\leq L\tr^{\frac{\bar q}{2}}.
\end{align*}
According to  \eqref{o2.9} and \eqref{o3.31} yields
\begin{align*}
\frac{p-\bar q}{p\tr^{\frac{\bar q^2}{2(p-\bar q)}}}\PP( \vartheta_{\tr}\leq T)
\leq &\frac{p-\bar q}{p\tr^{\frac{\bar q^2}{2(p-\bar q)}}}\Big(\PP(\delta_{\Gamma^{-1}(K\tr^{-\lambda})}\leq T)+\PP( \gamma_{\Gamma^{-1}(K\tr^{-\lambda}),\tr}\leq T)\Big)\nn\\
\leq &\frac{p-\bar q}{p\tr^{\frac{\bar q^2}{2(p-\bar q)}}}\frac{L_1+L_2}{\big(\Gamma^{-1}(K\tr^{-\lambda})\big)^p}.
\end{align*}
This, along with \eqref{o3.54} and \eqref{deterlambda} implies that
\begin{align*}
\frac{p-\bar q}{p\tr^{\frac{\bar q^2}{2(p-\bar q)}}}\PP( \vartheta_{\tr}\leq T)\leq L_3\tr^{\frac{\bar qp}{2(p-\bar q)}-\frac{\bar q^2}{2(p-\bar q)}}
=  L_3\tr^{\frac{\bar q}{2}},
\end{align*}
where the constant $L_3$ depends on $T$ but is independent of  $\tr$.
\end{proof}

\begin{lemma}\la{rateone}
Assume that $\textup{(A2)}$--$(\textup A5)$  hold with
$$ 3r+2\leq p
~~\hbox{and}~~
\bar q\in[2,\frac{2p}{3r+2}]\cap[2,\bar p).$$ Then for any $\tr\in(0,\bar \tr]$ and $T>0$,
\begin{align}\la{o3.59}
\E|x(T)-Z^{\tr}(T)|^{\bar q}\leq L_4\tr^{\frac{\bar q}{2}},
\end{align}
where the constant $L_4$ depends on $T$ but is independent of  $\tr$.
\end{lemma}
\begin{proof}\textbf{Proof.}
For any ~$\tr\in(0,\bar\tr]$ and $T>0$, by  Lemma \ref{ratetwo}  we get
\begin{align*}
\E|x(T)-Z^{\tr}(T)|^{\bar q}
=&\E\big(|x(T)-Z^{\tr}(T)|^{\bar q}\textup{\textbf 1}_{\{ \vartheta_{\tr}\leq T\}}\big)+\E\big(|x(T)-Z^{\tr}(T)|^{\bar q}\textup{\textbf 1}_{\{ \vartheta_{\tr}> T\}}\big)\nn\\
\leq& L_3\tr^{\frac{\bar q}{2} }+\E\big(|x(T)-Z^{\tr}(T)|^{\bar q}\textup{\textbf 1}_{\{ \vartheta_{\tr}>T\}}\big).
\end{align*}
We only require to show
\begin{align}\la{error+}
\E\big(|x(T)-Z^{\tr}(T)|^{\bar q}\textup{\textbf 1}_{\{ \vartheta_{\tr}>T\}}\big)\leq \hat L\tr^{\frac{\bar q}{2} },
\end{align}
where the constant $\hat L$ depends on $T$ but is independent of  $\tr$.
By virtue of \eqref{o1.1} and \eqref{o3.27}, we obtain
\begin{align*}
&x(T\wedge \vartheta_{\tr})-Z^{\tr}(T\wedge\vartheta_{\tr})\nn\\
=&\int_{0}^{T\wedge \vartheta_{\tr}}\big(f(x_t)-f(\bar Y_t^{\tr})\big)\mathrm dt+\int_{0}^{T\wedge\vartheta_{\tr}}\big(g(x_t)-g(\bar Y_t^{\tr})\big)\mathrm dB(t).
\end{align*}
Applying the It\^{o} formula  yields
\begin{align}\la{total}
&\E \big|x(T\wedge \vartheta_{\tr})-Z^{\tr}(T\wedge\vartheta_{\tr})\big|^{\bar q}\nn\\
\leq&\frac{\bar q}{2}\E\int_{0}^{T\wedge\vartheta_{\tr}}|x(t)-Z^{\tr}(t)|^{\bar q-2}
\Big(2\big(x(t)-Z^{\tr}(t)\big)^T \big(f(x_t)-f(\bar Y_t^{\tr})\big)\nn\\
&+(\bar q-1)|g(x_t)-g(\bar Y_t^{\tr})|^2\Big)\mathrm dt.
\end{align}
In view of $\bar q<\bar p$ and applying the Young inequality,  we derive
\begin{align}\la{fg+1}
&2\big(x(t)-Z^{\tr}(t)\big)^T\big(f(x_t)-f(\bar Y_t^{\tr})\big)+(\bar q-1)|g(x_t)-g(\bar Y_t^{\tr})|^2\nn\\
\leq&2\big(x(t)-Z^{\tr}(t)\big)^T \big(f(x_t)-f(Z_t^{\tr})\big)
+(\bar p-1)|g(x_t)-g(Z^{\tr}_t)|^2\nn\\
&+2|x(t)-Z^{\tr}(t)||f(Z_t^{\tr})-f(\bar Y_t^{\tr})|
+L|g(Z_t^{\tr})-g(\bar Y_t^{\tr})|^2.
\end{align}
Inserting  \eqref{fg+1} into \eqref{total} and using (A4)  leads to
\begin{align}\la{total2}
&\E \big|x(T\wedge \vartheta_{\tr})-Z^{\tr}(T\wedge\vartheta_{\tr})\big|^{\bar q}\nn\\
\leq&\frac{\bar qa_5}{2}\E\int_{0}^{T\wedge\vartheta_{\tr}}\big(|x(t)-Z^{\tr}(t)|^{\bar q}+\frac{1}{\tau}\int_{-\tau}^{0}|x_t(\theta)-Z^{\tr}_t(\theta)|^{\bar q}\mathrm d \theta\big)\mathrm dt\nn\\
&-\frac{\bar q}{2}\E\int_{0}^{T\wedge\vartheta_{\tr}}U(x(t),Z^{\tr}(t))\mathrm dt
+\frac{\bar q}{2}\E\int_{0}^{T\wedge\vartheta_{\tr}}
\int_{-\tau}^{0}U(x_t(\theta),Z^{\tr}_t(\theta))\rho_2(\theta)\mathrm d\theta\mathrm dt+I,
\end{align}
where
\begin{align*}
I:=&L\E\int_{0}^{T\wedge\vartheta_{\tr}}|x(t)-Z^{\tr}(t)|^{\bar q-2}\big(|x(t)-Z^{\tr}(t)||f(Z_t^{\tr})-f(\bar Y_t^{\tr})|+|g(Z_t^{\tr})-g(\bar Y_t^{\tr})|^2\big)\mathrm dt.
\end{align*}
By changing the integration order and using (A3) shows
\begin{align}\la{qterm}
&\frac{1}{\tau}\int_{0}^{T\wedge\vartheta_{\tr}}
\int_{-\tau}^{0}|x_t(\theta)-Z^{\tr}_t(\theta)|^{\bar q}\mathrm d\theta\mathrm dt
=\frac{1}{\tau}\!\int_{0}^{T\wedge\vartheta_{\tr}}
\int_{t-\tau}^{t}|x(s)-Z^{\tr}(s)|^{\bar q}\mathrm ds
\mathrm dt\nn\\
\leq&\frac{1}{\tau}\int_{-\tau}^{0}\int_{0}^{\tau+s}|x(s)-Z^{\tr}(s)|^{\bar q}\mathrm dt\mathrm ds+\frac{1}{\tau}\int_{0}^{T\wedge\vartheta_{\tr}}\int_{s}^{\tau+s}|x(s)-Z^{\tr}(s)|^{\bar q}\mathrm dt\mathrm ds\nn\\
\leq&\frac{1}{\tau}\int_{-\tau}^{0}\int_{0}^{\tau+s}|\xi(s)-
(\frac{t_{k_s+1}-s}{\tr}\xi(t_{k_s})+\frac{s-t_{k_s}}{\tr}\xi(t_{k_s+1}))
|^{\bar q}\mathrm dt\mathrm ds\nn\\
&+\frac{1}{\tau}\int_{0}^{T\wedge\vartheta_{\tr}}\int_{s}^{\tau+s}|x(s)-Z^{\tr}(s)|^{\bar q}\mathrm dt\mathrm ds\nn\\
\leq &L\tr^{\bar q/2}+\int_{0}^{T\wedge\vartheta_{\tr}}|x(s)-Z^{\tr}(s)|^{\bar q}\mathrm ds,
\end{align}
where $k_s:=\lfloor s/\tr\rfloor$.
Similarly, it follows from \eqref{calU}  and  (A3) that
\begin{align}\la{uterm}
\int_{0}^{T\wedge\vartheta_{\tr}}\int_{-\tau}^{0}
U(x_t(\theta),Z^{\tr}_t(\theta))\rho_2(\theta)\mathrm d\theta\mathrm dt
\leq L\tr^{\bar q/2}+\int_{0}^{T\wedge\vartheta_{\tr}}U(x(s),Z^{\tr}(s))\mathrm ds.
\end{align}
Inserting \eqref{qterm} and \eqref{uterm} into \eqref{total2} yields
\begin{align}\la{total3}
&\E \big|x(T\wedge \vartheta_{\tr})-Z^{\tr}(T\wedge\vartheta_{\tr})\big|^{\bar q}\nn\\
\leq& L\tr^{\bar q/2}+L\E\int_{0}^{T}|x(t\wedge\vartheta_{\tr})-Z^{\tr}(t\wedge\vartheta_{\tr})|^{\bar q}\mathrm dt+I,
\end{align}
We aim to estimate $I$.
Making use of the Young inequality we get
\begin{align}\la{Iterm2}
I
\leq&L\E\int_{0}^{T}|x(t\wedge\vartheta_{\tr})-Z^{\tr}(t\wedge\vartheta_{\tr})|^{\bar q}\mathrm dt
+L\E\int_{0}^{T}\big(|f(Z_t^{\tr})-f(\bar Y_t^{\tr})|^{\bar q}\nn\\
&+|g(Z_t^{\tr})-g(\bar Y_t^{\tr})|^{\bar q}\big)\mathrm dt.
\end{align}
By (\textup A5), the H\"{o}lder inequality, and $2\leq\bar q\leq 2p/(3r+2)$, we derive
\begin{align*}
&\E\big(|f(Z_t^{\tr})-f(\bar Y_t^{\tr})|^{\bar q}+|g(Z_t^{\tr})-g(\bar Y_t^{\tr})|^{\bar q}\big)\nn\\
\leq
&L\E\Big(\big(|Z^{\tr}_t(0)-\bar Y^{\tr}_t(0)|^{\bar q}
+\frac{1}{\tau}\int_{-\tau}^{0}|Z_t^{\tr}(\theta)-\bar Y_{t}^{\tr}(\theta)|^{\bar q}\mathrm d\theta\big)\nn\\
&\times\big(1+|Z^{\tr}(t)|^{\bar qr}+|\bar Y^{\tr}_t(0)|^{\bar qr}
+\frac{1}{\tau}\int_{-\tau}^{0}|Z_t^{\tr}(\theta)|^{\bar qr}\mathrm d\theta+\frac{1}{\tau}\int_{-\tau}^{0}|\bar Y_{t}^{\tr}(\theta)|^{\bar qr}\mathrm d\theta\big)\Big)\nn\\
\leq &L\Big(\E\big(|Z^{\tr}_t(0)-\bar Y^{\tr}_{t}(0)|^{\bar q}+\frac{1}{\tau}\!\int_{-\tau}^{0}|Z_t^{\tr}(\theta)-\bar Y_{t}^{\tr}(\theta)|^{\bar q}\mathrm d\theta\big)^{\frac{3r+2}{r+2}}\Big)^{\frac{r+2}{3r+2}}\Big(\E\big(1+|Z^{\tr}(t)|^{\bar qr}\nn\\
&+|\bar Y^{\tr}_t(0)|^{\bar qr}
+\frac{1}{\tau}\int_{-\tau}^{0}|Z^{\tr}_t(\theta)|^{\bar qr}\mathrm d\theta+\frac{1}{\tau}\int_{-\tau}^{0}|\bar Y_{t}^{\tr}(\theta)|^{\bar qr}\mathrm d\theta\big)^{\frac{3r+2}{2r}}\Big)^{\frac{2r}{3r+2}}\nn\\
\leq&L\Big(\E|Z^{\tr}_t(0)-\bar Y^{\tr}_t(0)|^{\frac{\bar q(3r+2)}{r+2}}
+\frac{1}{\tau}\int_{-\tau}^{0}\E|Z_t^{\tr}(\theta)-\bar Y_{t}^{\tr}(\theta)|^{\frac{\bar q(3r+2)}{r+2}}\mathrm d\theta\Big)^{\frac{r+2}{3r+2}} \Big(1+\E|Z^{\tr}(t)|^{p}\nn\\
&+\E|Y^{\tr}_t(0)|^{p}
+\frac{1}{\tau}\int_{-\tau}^{0}\E|Z^{\tr}_t(\theta)|^{p}\mathrm d\theta+\frac{1}{\tau}\int_{-\tau}^{0}\E|\bar Y_{t}^{\tr}(\theta)|^{p}\mathrm d\theta\Big) ^{\frac{2r}{3r+2}}.
\end{align*}
Using  Theorem \ref{oth2}, Lemma \ref{bounded}, and Lemma \ref{ztyt},  we derive
\begin{align*}
&\E\big(|f(Z_t^{\tr})-f(\bar Y_t^{\tr})|^{\bar q}
+|g(Z_t^{\tr})-g(\bar  Y_t^{\tr})|^{\bar q}\big)\nn\\
\leq& L\Big(\big(\E|Z^{\tr}_t(0)-\bar Y^{\tr}_t(0)|^{\frac{2p}{r+2}}\big)^\frac{\bar q(3r+2)}{2p}
+\frac{1}{\tau}\int_{-\tau}^{0}\big(\E|Z_t^{\tr}(\theta)-\bar Y_{t}^{\tr}(\theta)|^{\frac{2p}{r+2}}\big)^{\frac{\bar q(3r+2)}{2p}}\mathrm d\theta\Big)^{\frac{r+2}{3r+2}}\nn\\
\leq& L\tr^{\frac{\bar q}{2}}.
\end{align*}
Together with \eqref{Iterm2} we arrive at
\begin{align}\la{Iterm3}
I\leq L\E\int_{0}^{T}|x(t\wedge\vartheta_{\tr})-Z^{\tr}(t\wedge\vartheta_{\tr})|^{\bar q}\mathrm dt+L\tr^{\frac{\bar q}{2}}.
\end{align}
Inserting \eqref{Iterm3} into \eqref{total3}  and
\begin{align*}
\E \big|x(T\wedge \vartheta_{\tr})-Z^{\tr}(T\wedge\vartheta_{\tr})\big|^{\bar q}
\leq L\tr^{\frac{\bar q}{2}}+ L\E\int_{0}^{T}|x(t\wedge \vartheta_{\tr})-Z^{\tr}(t\wedge\vartheta_{\tr})|^{\bar q}\mathrm dt.
\end{align*}
The required assertion \eqref{error+} follows by applying the Gronwall inequality. 
\end{proof}

\begin{theorem}\la{thrate}
Assume that $\textup{(A2)}$--$(\textup A5)$  hold with $$3r+2\leq p
~~\hbox{and}~~
\bar q\in[2,\frac{2p}{3r+2}]\cap[2,\bar p).$$ Then for any  $\tr\in(0,\bar\tr]$ and $T>0$,
  \begin{align*}
  \E\big|x(T)-Y^{\tr}(T)\big|^{\bar q}
   \leq L_5\tr^{\frac{\bar q}{2} },
  \end{align*}
  where the constant $L_5$ depends on $T$ but is independent of  $\tr$.
 \end{theorem}
\begin{proof}\textbf{Proof.}
By virtue of  Lemma \ref{ztyt} and Lemma \ref{rateone} yields
\begin{align*}
  &\E|x(T)-Y^{\tr}(T)|^{\bar q}\nn\\
   \leq&L\E|x(T)-Z^{\tr}(T)|^{\bar q} +  L\E|Z^{\tr}_T(0)-\bar Y_T^{\tr}(0)|^{\bar q}\leq L_5 \tr^{\frac{\bar q}{2} },
\end{align*}
where the constant $L_5$ depends on $T$ but is independent of  $\tr$.
\end{proof}
\begin{remark}
In fact, the results  on the convergence and convergence rate
are still hold as $a_2=a_3=0$ in \textup{(A2)}.
\end{remark}
\section{Exponential stability}\la{ES}
In this section, we focus on approximating the exponential stability of SFDE \eqref{o1.1}.
We begin with the exponential stability
of the exact
solutions in $L^p$ and $\PP-1$. Then we approximate the long-time behaviors by the truncated numerical solutions.
Without loss of generality  we assume that $f(\mathbf{0})=0,~g(\mathbf{0})=0 $, and give the following hypothesis. 

\textbf{(A2$'$).}
 There exist positive  constants $p,~\varrho,~b_{i}~(1\leq i\leq4)$ satisfying $p\geq2$, ~$b_{1}>b_{2}, ~b_{3}>b_{4}$, and functions $\rho_3(\cdot), ~\rho_4(\cdot)\in \mathcal W$ such that for any $\phi\in \mathcal C$,
\begin{align}\label{A22}
&2(\phi(0))^T f(\phi)+(p-1)|g(\phi)|^2\nn\\
\leq& -b_{1}|\phi(0)|^{2}+b_{2}\int_{-\tau}^{0}|\phi(\theta)|^{2}\rho_3(\theta)\mathrm d\theta -b_{3}|\phi(0)|^{2+\varrho}+b_{4}\int_{-\tau}^{0}|\phi(\theta)|^{2+\varrho}\rho_4(\theta)\mathrm d\theta.
\end{align}

Owing to  $b_1>b_2$ and $b_3>b_4$ in $(\textup{A}2')$ , we can fix a positive constant $\nu$ such that
\begin{align}\la{o4.5}
\frac{p}{2}(b_{1}-b_{2}e^{\nu\tau})-\nu>0~~\hbox{and}~~
b_{3}-b_{4}e^{\nu\tau}>0.
\end{align}

By similar arguments in the proof of \cite[Theorem 3 and Theorem 4]{LMS2011} we give stability results of  the exact solutions.
\begin{theorem}\la{th2}
 Assume that $(\textup{A}1)$ and $(\textup{A}2')$ hold.
Then  the solution $x(t)$  to SFDE \eqref{o1.1} with the initial data $\xi\in \mathcal C$ satisfies
   \begin{align*}
   &\limsup_{t\rightarrow \infty}\frac{1}{t}\log(\E|x(t)|^{p})\leq -\nu,\nn\\
   &\limsup_{t\rightarrow \infty}\frac{1}{t}\log|x(t)| \leq - \frac{\nu}{p} ,~~~\hbox{a.s.}
   \end{align*}
  \end{theorem}

Now we start with  the stability analysis of the numerical solutions.
Due to \eqref{o4.5} we can fix a constant  $\e_1\in(0,1)$ sufficiently small such that
\begin{align}\la{se4+2}
&\frac{p}{2}\Big((1-\e_1)^{\frac{p}{2}}b_{1}-b_{2}e^{\nu\tau}\Big)
-\nu(1+\e_1e^{\nu\tau})\geq0,\nn\\
&(1-\e_1)^{\frac{p}{2}}b_{3}-b_{4}e^{\nu\tau}\geq0.
\end{align}
 For any $\phi\in \mathcal C$, define
\begin{align*}
\Phi_1(\phi):=(1-\e_1)|\phi(0)|^2
+\frac{\e_1}{\tau}\int_{-\tau}^{0}|\phi(\theta)|^2\mathrm{d}\theta.
\end{align*}
In the similar ways as  \eqref{o3.6} and \eqref{o3.6+} were shown, we get that for any  $k\in\mathbb N$,
 \begin{align}\la{o4.10}
 &\big|f(\bar Y_{t_k}^{\tr})\big|^2
\leq L\tr^{-2\lambda}\Phi_1(\bar Y_{t_k}^{\tr}),
\end{align}
\begin{align}\la{o4.10+}
 &\big|g(\bar Y^{\tr}_{t_k})\big|^2
\leq L\tr^{-\lambda}\Phi_1(\bar Y_{t_k}^{\tr}).
\end{align}
The following inequality  plays a key role in the proof of the stability of the numerical solutions.
 \begin{lemma}\label{+r2}
 Assume that  $(\textup{A}2')$ holds.
  Then for any $\phi\in \mathcal C$,
\begin{align}\la{re2+3}
&\nu\big(\Phi_1(\phi)\big)^{\frac{p}{2}}+\frac{(1- \e_1) p}{2}\big(\Phi_1(\phi)\big)^{\frac{p-2}{2}}\big(2(\phi(0))^T f(\phi)+(p-1)|g(\phi)|^2\big)\nn\\
\leq&-\beta_1|\phi(0)|^{ p}+\frac{\beta_2}{\tau}\int_{-\tau}^{0}|\phi(\theta)|^{p}\mathrm d\theta+\beta_3\int_{-\tau}^{0}|\phi(\theta)|^{ p}\rho_3(\theta)\mathrm d\theta\nn\\
&-\beta_{4}|\phi(0)|^{ p+\varrho}+\frac{\beta_5}{\tau}\int_{-\tau}^{0}|\phi(\theta)|^{ p+\varrho}\mathrm d\theta+\beta_6\int_{-\tau}^{0}|\phi(\theta)|^{ p+\varrho}\rho_4(\theta)\mathrm d\theta,
\end{align}
where  $\beta_{i}~(1\leq i\leq6)$ are positive constants  given by \eqref{be},
and satisfy
$$e^{\nu\tau}(\b_2+\b_3)<\b_1~~\hbox{and}~~ e^{\nu\tau}(\b_5+\b_6)<\b_4.$$
 \end{lemma}
\begin{proof}\textbf{Proof.}
For any $\phi\in \mathcal C$, making use of \eqref{A22} and the Young inequality leads to
\begin{align}\la{l4.1+1}
 &\big(\Phi_1(\phi)\big)^{\frac{p-2}{2}}\big(2(\phi(0))^T f(\phi)+(p-1)|g(\phi)|^2\big)\nn\\
\leq&-(1-\e_1)^{\frac{p-2}{2}}b_1|\phi(0)|^{ p}+b_2\big(\Phi_1(\phi)\big)^{\frac{ p-2}{2}}\int_{-\tau}^{0}|\phi(\theta)|^{2}\rho_3(\theta)\mathrm d\theta\nn\\
&-(1-\e_1)^{\frac{p-2}{2}}b_3|\phi(0)|^{ p+\varrho}+b_4\big(\Phi_1(\phi)\big)^{\frac{ p-2}{2}}\int_{-\tau}^{0}|\phi(\theta)|^{2+\varrho}\rho_4(\theta)\mathrm d\theta\nn\\
\leq&-(1-\e_1)^{\frac{p-2}{2}}b_1|\phi(0)|^{p}+\frac{(p-2)b_2}{ p}\big(\Phi_1(\phi)\big)^{\frac{p}{2}}+\frac{2b_2}{p}\big(\int_{-\tau}^{0}|\phi(\theta)|^{2}\rho_3(\theta)\mathrm d\theta\big)^{\frac{p}{2}}\nn\\
&-(1-\e_1)^{\frac{p-2}{2}}b_3|\phi(0)|^{p+\varrho}+\frac{(p-2)b_4}{ p+\varrho}\big(\Phi_1(\phi)\big)^{\frac{p+\varrho}{2}}\nn\\
&+\frac{(2+\varrho) b_4}{ p+\varrho}\big(\int_{-\tau}^{0}|\phi(\theta)|^{2+\varrho}\rho_4(\theta)\mathrm d\theta\big)^{\frac{p+\varrho}{2+\varrho}}.
\end{align}
Applying the H\"{o}lder inequality yields
\begin{align}\la{l4.1+2}
&\big(\int_{-\tau}^{0}|\phi(\theta)|^{2}\rho_3(\theta)\mathrm d\theta\big)^{\frac{p}{2}}\leq \int_{-\tau}^{0}|\phi(\theta)|^{p}\rho_3(\theta)\mathrm d\theta\nn\\
&\big(\int_{-\tau}^{0}|\phi(\theta)|^{2+\varrho}\rho_4(\theta)\mathrm d\theta\big)^{\frac{ p+\varrho}{2+\varrho}}\leq\int_{-\tau}^{0}|\phi(\theta)|^{ p+\varrho}\rho_4(\theta)\mathrm d\theta.
\end{align}
Combining \eqref{l4.1+1} and \eqref{l4.1+2} we derive
\begin{align}\la{l4.1+4}
&\nu\big(\Phi_1(\phi)\big)^{\frac{p}{2}}+\frac{(1- \e_1) p}{2}\big(\Phi_1(\phi)\big)^{\frac{p-2}{2}}\big(2(\phi(0))^T f(\phi)+(p-1)|g(\phi)|^2\big)\nn\\
\leq&-\frac{(1-\e_1)^{\frac{p}{2}}pb_1}{2}|\phi(0)|^{p}
+(\nu+\frac{(p-2)b_2}{2})\big(\Phi_1(\phi)\big)^{\frac{p}{2}}\nn\\
&+b_2\int_{-\tau}^{0}|\phi(\theta)|^{p}\rho_3(\theta)\mathrm d\theta-\frac{(1-\e_1)^{\frac{p}{2}}pb_3}{2}|\phi(0)|^{p+\varrho}\nn\\
&+\frac{(p-2)pb_4}{2(p+\varrho)}
\big(\Phi_1(\phi)\big)^{\frac{p+\varrho}{2}}+\frac{(2+\varrho)pb_4}{2(p+\varrho)}
\int_{-\tau}^{0}|\phi(\theta)|^{p+\varrho}\rho_4(\theta)\mathrm d\theta,
\end{align}
where the positive constant $\nu$ is given by \eqref{o4.5}.
Using the convex property of $u(x)=x^{a}~(a>1)$  and the H\"{o}lder inequality yields
\begin{align}\la{l4.1+3}
\big(\Phi_1(\phi)\big)^{\frac{p}{2}}
\leq&(1-\e_1)|\phi(0)|^{p}+\e_1(\frac{1}{\tau}\int_{-\tau}^{0}|\phi(\theta)|^{2}\mathrm d\theta)^{\frac{p}{2}}\nn\\
\leq&(1-\e_1)|\phi(0)|^{p}+\frac{\e_1}{\tau}\int_{-\tau}^{0}|\phi(\theta)|^{ p}\mathrm d\theta,\nn\\
\big(\Phi_1(\phi)\big)^{\frac{p+\varrho}{2}}
\leq&(1-\e_1)|\phi(0)|^{p+\varrho}
+\frac{\e_1}{\tau}\int_{-\tau}^{0}|\phi(\theta)|^{p+\varrho}\mathrm d\theta.
\end{align}
Inserting  \eqref{l4.1+3} into \eqref{l4.1+4}  yields
\begin{align*}
&\nu\big(\Phi_1(\phi)\big)^{\frac{p}{2}}+\frac{(1- \e_1) p}{2}\big(\Phi_1(\phi)\big)^{\frac{p-2}{2}}\big(2(\phi(0))^T f(\phi)+(p-1)|g(\phi)|^2\big)\nn\\
\leq &-\b_1|\phi(0)|^{ p}+\frac{\b_2}{\tau}\int_{-\tau}^{0}|\phi(\theta)|^{p}\mathrm d\theta+\b_3\int_{-\tau}^{0}|\phi(\theta)|^{ p}\rho_3(\theta)\mathrm d\theta\nn\\
&-\b_4|\phi(0)|^{p+\varrho}+\frac{\b_5}{\tau}\int_{-\tau}^{0}|\phi(\theta)|^{p+\varrho}\mathrm d\theta+\b_6\int_{-\tau}^{0}|\phi(\theta)|^{ p+\varrho}\rho_4(\theta)\mathrm d\theta,
\end{align*}
where
\begin{align}\la{be}
&\b_1:=\frac{(1-\e_1)^{\frac{p}{2}} pb_1}{2}-(1-\e_1)(\nu+\frac{(p-2)b_2}{2}),\nn\\
&\b_2:=\e_1(\nu+\frac{( p-2)b_2}{2}),~~~\b_3:=b_2,\nn\\
&\b_4:=\frac{p}{2}\big((1-\e_1)^{\frac{p}{2}}b_3-\frac{(1-\e_1)(p-2)b_4}{p+\varrho}\big),\nn\\
&\b_5:=\frac{(p-2)pb_4\e_1}{2( p+\varrho)},~~~\b_6:=\frac{(2+\varrho)pb_4}{2( p+\varrho)}.
 \end{align}
It follows from \eqref{se4+2} that
$$e^{\nu\tau}(\b_2+\b_3)<\b_1~~\hbox{and}~~ e^{\nu\tau}(\b_5+\b_6)<\b_4.$$
The proof is complete.
\end{proof}
 \begin{theorem}\la{th6}
Assume that $(\textup{A}1)$ and  $(\textup{A}2')$ hold. Then for any  $\nu_1\in(0,\nu)$, there exists a $\hat\tr\in(0,1]$ such that for any $\tr\in(0,\hat \tr]$,  
 \be\la{o4.11+1}
   \limsup_{k\rightarrow\infty}\frac{1}{k\tr}\log(\E|Y^{\tr}(t_k)|^{p}) \leq -\nu_1,
   \ee
  \begin{align}\la{o4.11++1}
   \limsup_{k\rightarrow\infty}\frac{1}{k\tr}\log{|Y^{\tr}(t_k)|} \leq- \frac{\nu_1}{p}, ~\hbox{a.s.}
   \end{align}

  \end{theorem}
\begin{proof}\textbf{Proof.}
This proof uses the similar techniques as  Theorem \ref{oth2}.
For any $\iota\in(0,1]$ and  $i\in\mathbb N$, in a similar way as \eqref{o3.13} was shown, we derive
\begin{align}\label{o4.13}
\E \Big(\big(\iota+\Phi_1(\bar Y^{\tr}_{t_{i+1}})\big)^{\frac{ p}{2}}\big|\mathcal{F}_{t_{i}}\Big)
\leq &\E\Big(\big(\iota+(1-\e_1)|\breve Y^{\tr}(t_{i+1})|^2+\frac{\e_1}{\tau}\int_{-\tau}^{0}|\bar Y^{\tr}_{t_{i+1}}(\theta)|^2\mathrm d\theta\big)^{\frac{ p}{2}}\big|\mathcal{F}_{t_{i}}\Big)\nn\\
\leq&\big(\iota+\Phi_1(\bar Y^{\tr}_{t_i})\big)^{\frac{p}{2}}\Big(1 + \frac{p}{2} \E\big(\hat\Theta_{i}\big|\mathcal{F}_{t_{i}}\big)+ \frac{p( p-2)}{8} \E\big(\hat \Theta_{i}^2\big|\mathcal{F}_{t_{i}}\big)\nn\\
    &+  \E\big(\hat \Theta_{i}^3P_{\ell}(\hat \Theta_{i})\big|\mathcal{F}_{t_{i}}\big)\Big),
\end{align}
where the integer $\ell$ satisfies $2\ell<p\leq 2(\ell+1)$, $P_{\ell}(\cdot)$ is an $\ell$th-order polynomial, and
\begin{align*}
\hat \Theta_{i}=&\big(\iota+\Phi_1(\bar Y^{\tr}_{t_i})\big)^{-1}\Big((1-\e_1)\big(|f(\bar Y^{\tr}_{t_i})|^2\tr^2+|g(\bar Y^{\tr}_{t_i}) \tr B_{i}|^2+2(Y^{\tr}(t_i))^T f(\bar Y^{\tr}_{t_i})\tr\nn\\
&+2(Y^{\tr}(t_i))^T g(\bar Y^{\tr}_{t_i})\tr B_{i}
+2 f^T(\bar Y^{\tr}_{t_i})g(\bar Y^{\tr}_{t_i}) \tr B_{i}\tr\big)+\e_1\big(\frac{2\tr}{\tau}|Y^{\tr}(t_i)|^2\nn\\
&+\frac{3\tr^3}{2\tau}|f(\bar Y^{\tr}_{t_i})|^2+ \frac{3\tr}{2\tau}|g(\bar Y^{\tr}_{t_i}) \tr B_{i}|^2\big)
\Big).
\end{align*}
In a similar way as  Lemma  \ref{esti} was shown, and making use of \eqref{o4.10} and   \eqref{o4.10+}, we obtain
\begin{align}\la{th4.2+1}
 &\E\big(\hat \Theta_{i}\big|\mathcal{F}_{t_{i}}\big)
\leq (1-\e_1)\big(\iota+\Psi_1(\bar Y^{\tr}_{t_i})\big)^{-1}\big(2(Y^{\tr}(t_i))^T f(\bar Y^{\tr}_{t_i})+|g(\bar Y^{\tr}_{t_i})|^2\big)\tr+L\tr^{2-2\lambda},\nn\\
&\E\big(\hat\Theta_{i}^2\big|\mathcal{F}_{t_{i}}\big)
\leq 4(1-\e_1)^2\big(\iota+\Psi_1(\bar Y^{\tr}_{t_i})\big)^{-2}\big|(Y^{\tr}(t_i))^T g(\bar Y^{\tr}_{t_i})\big|^2\tr+L\tr^{2-2\lambda},\nn\\
&\E\big(\hat \Theta_{i}^3P_{\ell}(\hat \Theta_{i})\big|\mathcal{F}_{t_{i}}\big)
\leq L\tr^{2-2\lambda}.
\end{align}
Inserting \eqref{th4.2+1} into \eqref{o4.13} and taking expectations leads to
\begin{align}\la{mono}
&\E\big(\iota+\Phi_1(\bar Y^{\tr}_{t_{i+1}})\big)^{\frac{p}{2}}
=\E \Big(\E\big((\iota+\Phi_1(\bar Y^{\tr}_{t_{i+1}}))^{\frac{ p}{2}}\big|\mathcal{F}_{t_{i}}\big)\Big)\nn\\
\leq& (1+L\tr^{2-2\lambda})\E\big(\iota+\Phi_1(\bar Y^{\tr}_{t_i})\big)^{\frac{p}{2}}+\frac{(1-\e_1) p\tr}{2}\E\Big(
(\iota+\Phi_1(\bar Y^{\tr}_{t_i}))^{\frac{ p-4}{2}}\big((\iota+\Phi_1(\bar Y^{\tr}_{t_i}))\nn\\
&\times\big(2(Y^{\tr}(t_i))^Tf(\bar Y^{\tr}_{t_i})+|g(\bar Y^{\tr}_{t_i})|^2\big)+(1-\e_1)(p-2)|(Y^{\tr}(t_i))^Tg(\bar Y^{\tr}_{t_i})|^2 \big)\Big)\nn\\
\leq& (1+L\tr^{2-2\lambda})\E\big(\iota+\Phi_1(\bar Y^{\tr}_{t_i})\big)^{\frac{p}{2}}+\frac{(1-\e_1) p\tr}{2}\E\Big(\big(\iota+\Phi_1(\bar Y^{\tr}_{t_i})\big)^{\frac{ p-2}{2}}\big(2(Y^{\tr}(t_i))^Tf(\bar Y^{\tr}_{t_i})\nn\\
&+(p-1)|g(\bar Y^{\tr}_{t_i})|^2\big)\Big).
\end{align}
It is straightforward to see from $\iota\in(0,1]$ that
$$
\big(\iota+\Phi_1(\bar Y^{\tr}_{t_{i+1}})\big)^{\frac{p}{2}}\leq \big(1+\Phi_1(\bar Y^{\tr}_{t_{i+1}})\big)^{\frac{p}{2}},$$
and
\begin{align*}
&\big(\iota+\Phi_1(\bar Y^{\tr}_{t_i})\big)^{\frac{ p-2}{2}}\big(2(Y^{\tr}(t_i))^Tf(\bar Y^{\tr}_{t_i})+(p-1)|g(\bar Y^{\tr}_{t_i})|^2\big)\tr\nn\\
\leq &\big(1+\Phi_1(\bar Y^{\tr}_{t_{i}})\big)^{\frac{p-2}{2}}(2|Y^{\tr}(t_i)||f(\bar Y^{\tr}_{t_i})|+(p-1)|g(\bar Y^{\tr}_{t_i})|^2)\tr.
\end{align*}
By the definition of $\Phi_1$ and Theorem \ref{oth2}  we have
\begin{align*}
\E\big(1+\Phi_1(\bar Y^{\tr}_{t_{i+1}})\big)^{\frac{p}{2}}\leq L\big(1+\sup_{\tr\in(0,1]}\sup_{j\geq-N}\E|Y^{\tr}(t_{j})|^{p}\big)\leq L.
\end{align*}
By  \eqref{o4.10}, \eqref{o4.10+}, and Theorem \ref{oth2}  we derive
\begin{align*}
&\E\big(1+\Phi_1(\bar Y^{\tr}_{t_{i}})\big)^{\frac{p-2}{2}}(2|Y^{\tr}(t_i)||f(\bar Y^{\tr}_{t_i})|+(p-1)|g(\bar Y^{\tr}_{t_i})|^2)\tr\nn\\
\leq &L\tr^{1-\lambda}\E\big(1+\Phi_1(\bar Y^{\tr}_{t_{i}})\big)^{\frac{p}{2}}\nn\\
\leq&L\big(1+\sup_{\tr\in(0,1]}\sup_{j\geq-N}\E|Y^{\tr}(t_{j})|^{p}\big)\leq L.
\end{align*}
Hence according to \eqref{mono} and  using the dominated convergence theorem
yields
\begin{align*}
&\E\big(\Phi_1(\bar Y^{\tr}_{t_{i+1}})\big)^{\frac{p}{2}}
=\lim_{\iota\rightarrow 0^+}\E\big(\iota+\Phi_1(\bar Y^{\tr}_{t_{i+1}})\big)^{\frac{p}{2}}\nn\\
\leq& (1+L\tr^{2-2\lambda})\E\big(\Phi_1(\bar Y^{\tr}_{t_i})\big)^{\frac{p}{2}}+\frac{(1-\e_1)p\tr}{2}\E\Big(\big(\Phi_1(\bar Y^{\tr}_{t_i})\big)^{\frac{p-2}{2}}\big(2(Y^{\tr}(t_i))^Tf(\bar Y^{\tr}_{t_i})\nn\\
&+(p-1)|g(\bar Y^{\tr}_{t_i})|^2\big)\Big).
\end{align*}
For any $\nu_1\in(0,\nu)$, choose a  $\hat\tr\in(0,1]$  sufficiently  small such that $$L\hat\tr^{1-2\lambda}+\nu_1\leq\nu.$$
Making use of $
e^{\nu_1 t_{i+1}}-e^{\nu_1 t_i}
\leq e^{\nu_1 t_{i+1}}\nu_1\tr$ we derive that for any $\tr\in(0,\hat\tr]$,
\begin{align*}
&e^{\nu_1 t_{i+1}}\E\big(\Phi_1(\bar Y^{\tr}_{t_{i+1}})\big)^{\frac{p}{2}}\nn\\
\leq&(e^{\nu_1 t_i}+e^{\nu_1 t_{i+1}}\nu_1\tr)\E\big(\Psi_1(\bar Y^{\tr}_{t_i})\big)^{\frac{p}{2}}+e^{\nu_1 t_{i+1}}\tr\E\Big(L\hat\tr^{1-2\lambda}\big(\Phi_1(\bar Y^{\tr}_{t_i})\big)^{\frac{p}{2}}+\frac{(1-\e_1)p}{2}\nn\\
&\times\big(\Phi_1(\bar Y^{\tr}_{t_i})\big)^{\frac{p-2}{2}}\big(2(Y^{\tr}(t_i))^Tf(\bar Y^{\tr}_{t_i})+(p-1)|g(\bar Y^{\tr}_{t_i})|^2\big)\Big)\nn\\
\leq &e^{\nu_1 t_i}\E\big(\Phi_1(\bar Y^{\tr}_{t_i})\big)^{\frac{p}{2}}+e^{\nu_1 t_{i+1}}\tr\E\Big(\nu\big(\Phi_1(\bar Y^{\tr}_{t_i})\big)^{\frac{ p}{2}}+\frac{(1-\e_1)p}{2}\big(\Phi_1(\bar Y^{\tr}_{t_i})\big)^{\frac{ p-2}{2}}\nn\\
&\times\big(2(Y^{\tr}(t_i))^Tf(\bar Y^{\tr}_{t_i})+(p-1)|g(\bar Y^{\tr}_{t_i})|^2\big)\Big).
\end{align*}
This, along with \eqref{re2+3} implies that
\begin{align*}
&e^{\nu_1 t_{i+1}}\E\big(\Phi_1(\bar Y^{\tr}_{t_{i+1}})\big)^{\frac{p}{2}}-e^{\nu_1 t_i}\E\big(\Phi_1(\bar Y^{\tr}_{t_i})\big)^{\frac{p}{2}}\nn\\
\leq&e^{\nu_1  t_{i+1}}\tr\E\Big(-\b_1| Y^{\tr}(t_i)|^{p}+\frac{ \b_2}{\tau}\int_{-\tau}^{0}|\bar  Y^{\tr}_{t_{i}}(\theta)|^{p}\mathrm d\theta+\b_3\int_{-\tau}^{0}|\bar  Y^{\tr}_{t_{i}}(\theta)|^{p}\rho_3(\theta)\mathrm d\theta\nn\\
&-\b_4| Y^{\tr}(t_i)|^{p+\varrho}+\frac{\b_5}{\tau}\int_{-\tau}^{0}|\bar  Y^{\tr}_{t_{i}}(\theta)|^{p+\varrho}\mathrm d\theta
+\b_6\int_{-\tau}^{0}|\bar  Y^{\tr}_{t_{i}}(\theta)|^{p+\varrho}\rho_4(\theta)\mathrm d\theta\Big).
\end{align*}
Summing the above inequality on both sides from $i=0$ to $k$ yields
\begin{align}\la{totalstabnu}
&e^{\nu_1 t_{k+1}}\E\big(\Phi_1(\bar Y^{\tr}_{t_{k+1}})\big)^{\frac{p}{2}}\nn\\
\leq&\big(\Phi_1(\bar Y^{\tr}_{0})\big)^{\frac{p}{2}}+\tr\sum_{i=0}^{k}e^{\nu_1  t_{i+1}}\E\Big(-\b_1|Y^{\tr}(t_i)|^{p}+\frac{\b_2}{\tau}\int_{-\tau}^{0}|\bar Y^{\tr}_{t_{i}}(\theta)|^{p}\mathrm d\theta\nn\\
&+\b_3\int_{-\tau}^{0}|\bar Y^{\tr}_{t_{i}}(\theta)|^{p}\rho_3(\theta)\mathrm d\theta-\b_4|Y^{\tr}(t_i)|^{p+\varrho}\nn\\
&+\frac{\b_5}{\tau}\int_{-\tau}^{0}|\bar Y^{\tr}_{t_{i}}(\theta)|^{p+\varrho}\mathrm d\theta+\b_6\int_{-\tau}^{0}|\bar Y^{\tr}_{t_{i}}(\theta)|^{p+\varrho}\rho_4(\theta)\mathrm d\theta\Big).
\end{align}
In the similar ways as \eqref{o3.26} and \eqref{o3.26+1} were shown, we have
\begin{align}\la{b22n}
&\sum_{i=0}^{k}e^{\nu_1 t_{i+1}}\big(\frac{1}{\tau}\int_{-\tau}^{0}|\bar Y^{\tr}_{t_{i}}(\theta)|^{p}\mathrm d\theta\big)
\leq L+e^{\nu_1\tau}\sum_{i=0}^{k}e^{\nu_1  t_{i+1}}|Y^{\tr}(t_i)|^{p},\nn\\
&\sum_{i=0}^{k}e^{\nu_1 t_{i+1}}\int_{-\tau}^{0}|\bar Y^{\tr}_{t_{i}}(\theta)|^{p}\rho_3(\theta)\mathrm d\theta
\leq L+e^{\nu_1\tau}\sum_{i=0}^{k}e^{\nu_1  t_{i+1}}|Y^{\tr}(t_i)|^{p},\nn\\
&\sum_{i=0}^{k}e^{\nu_1 t_{i+1}}\big(\frac{1}{\tau}\int_{-\tau}^{0}|\bar Y^{\tr}_{t_{i}}(\theta)|^{p+\varrho}\mathrm d\theta\big)
\leq L+e^{\nu_1\tau}\sum_{i=0}^{k}e^{\nu_1  t_{i+1}}|Y^{\tr}(t_i)|^{p+\varrho},\nn\\
&\sum_{i=0}^{k}e^{\nu_1 t_{i+1}}\int_{-\tau}^{0}|\bar Y^{\tr}_{t_{i}}(\theta)|^{p+\varrho}\rho_4(\theta)\mathrm d\theta
\leq L+e^{\nu_1\tau}\sum_{i=0}^{k}e^{\nu_1  t_{i+1}}|Y^{\tr}(t_i)|^{p+\varrho}.
\end{align}
Inserting \eqref{b22n} into \eqref{totalstabnu} and then using $$e^{\nu\tau}(\b_2+\b_3)< \b_1~~\hbox{and}~~ e^{\nu\tau}(\b_5+\b_6)<\b_4,$$
we derive
 \begin{align*}
e^{\nu_1 t_{k+1}}\E\big(\Phi_1(\bar Y^{\tr}_{t_{k+1}})\big)^{\frac{p}{2}}
\leq&L,
\end{align*}
which implies that (\ref{o4.11+1}) holds.
By virtue of (\ref{o4.11+1}), using the similar technique as  in the proof of \cite[Theorem 3.4]{wu_mao-kloeden2013} implies the desired assertion \eqref{o4.11++1}.
\end{proof}
\begin{rem}
The results of exponential stability for the exact
solutions and  the numerical solutions still hold as $b_3=b_4=0$ in \textup{(A2$'$)}.
\end{rem}
\begin{rem}
Our numerical method is also suitable for systems with coefficients $f(\cdot)$ and $g(\cdot)$ to be of the form
$$H(\phi(0),\int_{-\tau_1}^{0}\phi(\theta)\rho_{1}(\theta)\mathrm d\theta,\cdots,\int_{-\tau_M}^{0}\phi(\theta)\rho_{M}(\theta)\mathrm d\theta),$$
where $M$ is a positive integer, ~$\tau_i>0$ and $\rho_i\in\mathcal W([-\tau_{i},0];\RR_+)$ for any $i\in\{1,\cdots,M\}$, ~$\tau:=\max_{1\leq i\leq M}\tau_i$ and $\phi\in C([-\tau,0];\RR^n)$, and the function $H:\RR^n\times\cdots\times\RR^n\rightarrow \RR^n$ satisfies the local Lipschitz condition.
In such a case,
we may rewrite  $\sum_{i=1}^{M}\int_{-\tau_i}^{0}\phi(\theta)\rho_i(\theta)\mathrm d\theta$
as
 $\int_{-\tau}^{0}\phi(\theta)\bar\rho(\theta)\mathrm d\theta,$
where $  \bar\rho(\theta):=\frac{1}{M}\sum_{i=1}^{M}\rho_i(\theta)\textup{\textbf1}_{[-\tau_i,0]}(\theta) \in\mathcal W([-\tau,0];\RR_+).$
\end{rem}

\section{Numerical examples}\label{Nexamp}
To verify the efficiency of our explicit scheme  we give  two examples and some numerical experiments.
\begin{expl}
{\rm
Recall SFDE \eqref{e1} and let $p=8$.
Applying the H\"{o}lder inequality yields
\begin{align*}
&2\langle f(\phi),\phi(0)\rangle+(p-1)|g(\phi)|^2
\leq2\phi(0)+8\phi^2(0)-8\phi^4(0)+7\int_{-1/2}^{0}\phi^4(\theta)\cdot2\mathrm d\theta,
\end{align*}
which implies that (A2) holds.
It is straightforward to see that (A3)  holds.
Let $\bar q=2$, $\bar p=3$, and compute
\begin{align*}
&2\big(\phi(0)-\bar\phi(0)\big)^T\big(f(\phi)-f(\bar \phi)\big)+(\bar p-1)|g(\phi)-g(\bar\phi)|^2\nn\\
\leq&8|\phi(0)-\bar\phi(0)|^2
-8|\phi(0)-\bar\phi(0)|^2\big(\phi^2(0)+\phi(0)\bar\phi(0)+\bar\phi^2(0)\big)\nn\\
&+2\int_{-1/2}^{0}|\phi(\theta)-\bar\phi(\theta)|^2|\phi(\theta)+\bar\phi(\theta)|^2\cdot2\mathrm d\theta\nn\\
\leq&8|\phi(0)-\bar\phi(0)|^2-4|\phi(0)-\bar\phi(0)|^2|\phi(0)+\bar\phi(0)|^2\nn\\
&+2\int_{-1/2}^{0}|\phi(\theta)-\bar\phi(\theta)|^2|\phi(\theta)+\bar\phi(\theta)|^2\cdot 2\mathrm d\theta,
\end{align*}
which implies that (A4) holds.
By a direct computation we know  that (A5) holds with $a_6=6$ and $r=2$.
In view of  Remark \ref{r2}, we take
$\Gamma(l)=6\sqrt2(1+4l^2)$ for any $l\geq1$. This implies
$$\Gamma^{-1}(l)=(\frac{l}{24\sqrt2}-\frac{1}{4})^{1/2}, ~~~\forall ~l\geq 30\sqrt2.$$
By \eqref{deterlambda} we have $\lambda=1/3$.
 Together with \eqref{tru} one goes a further step to obtain that for any $\tr\in(0,1]$ and $x\in\RR^n$,
\begin{align*}
\Lambda^{\triangle,\lambda}_{\Gamma}(x)= \Big(|x|\wedge \big(\frac{5}{4\tr^{1/3}}-\frac{1}{4}\big)^{1/2}\Big) \frac{x}{|x|}.
\end{align*}
Let $\bar\tr=2^{-5}$.
By virtue of Theorem \ref{thrate},  truncated EM solution \eqref{o3.8} has the property that for any $\tr\in(0,2^{-5}]$ and $T>0$,
\begin{align*}
\big(\E\big|x(T)-Y^{\tr}(T)\big|^{2}\big)^{1/2}
   \leq L \tr^{1/2}.
\end{align*}
Since it is impossible to find the closed form of the solutions of SFDE \eqref{e1}, we have to regard the  truncated EM solution  $\{Y^{\tr}(t)\}_{t\geq-\tau}$ with the smaller stepsize $\tr=2^{-18}$ as the
exact solution $\{x(t)\}_{t\geq-\tau}$.
To verify the convergence of the truncated EM solution $\{Y^{\tr}(t)\}_{t\geq-\tau}$, we simulate the root mean square approximation error $\big(\E\big|x(10)-Y^{\tr}(10)\big|^{2}\big)^{1/2}$ by using MATLAB soft.
In Figure \ref{sfdeerror},  the red dotted line depicts $\big(\E\big|x(10)-Y^{\tr}(10)\big|^{2}\big)^{1/2}$ as the function of $\tr$ for 1000 sample points as $\tr\in\{2^{-5},2^{-6},2^{-7},2^{-8},2^{-10}\}$, while
the blue line represents the reference line with the slope $1/2$.
\begin{figure}[htp]
  \begin{center}
\includegraphics[width=12cm,height=8cm]{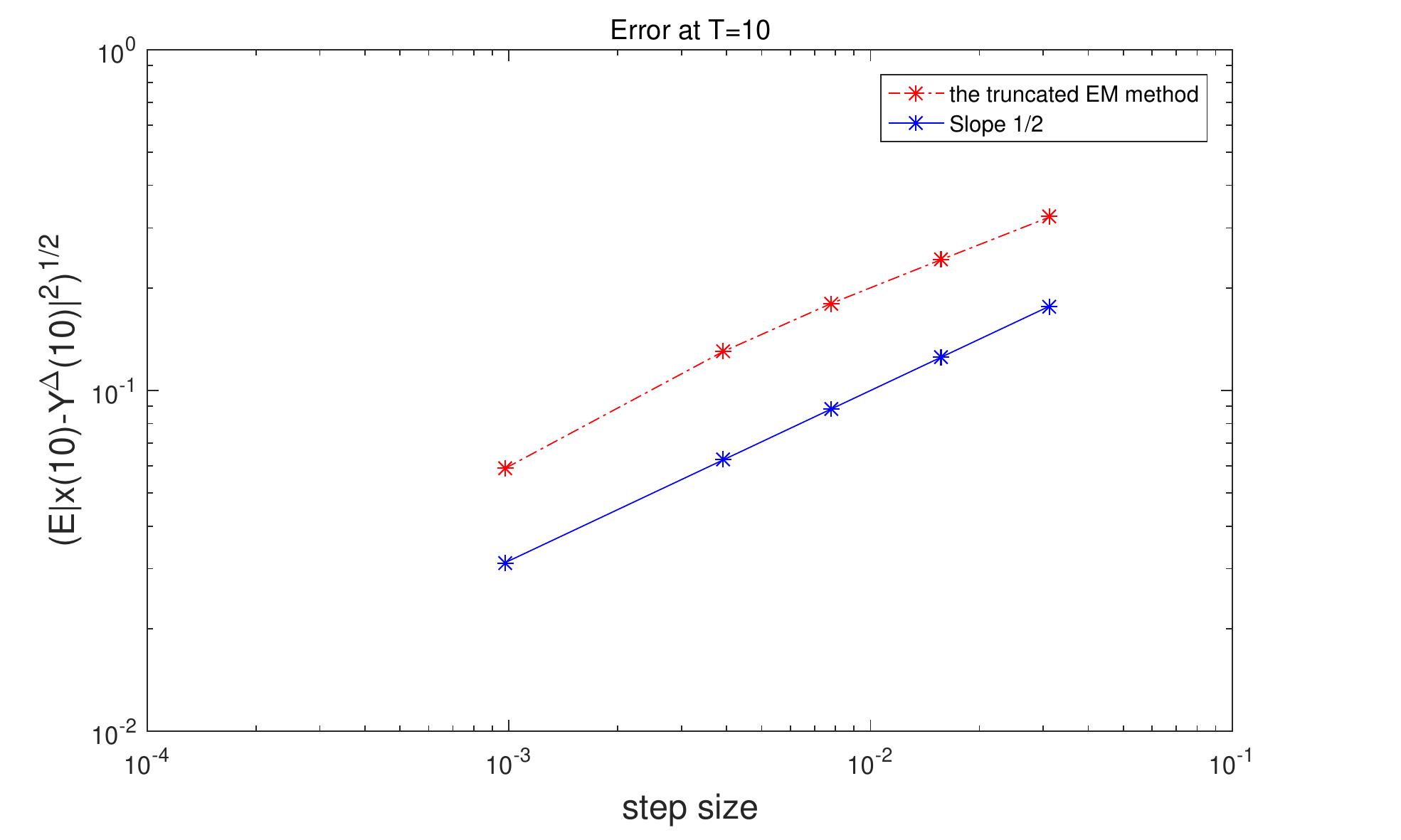}
   \end{center}
 \caption{The red dotted line depicts
 the root mean square approximation error $\big(\E\big|x(10)-Y^{\tr}(10)\big|^{2}\big)^{1/2}$
between  exact solution $x(10)$ and  truncated EM solution $Y^{\tr}(10)$
 for 1000 sample points as $\tr\in\{2^{-5},2^{-6},2^{-7},2^{-8},2^{-10}\}$, while the blue line represents the reference line with the slope $1/2$.  }
\label{sfdeerror}
\end{figure}
}
\end{expl}

\begin{expl}
\rm{
Consider the $2$-dimensional  SFDE
\begin{align}\la{e2}
\left\{
\begin{array}{ll}
\mathrm dx_1(t)=\big(-2x_1(t)-3x_1^3(t)+\int_{-1/4}^{0}x_2(t+\theta)\mathrm d\theta\big)\mathrm dt
+x_1(t)\mathrm dB_1(t),~~~t>0,\\
\mathrm dx_2(t)=\big(-2x_2(t)-2x_2^3(t)+\int_{-1/2}^{0}x^3_1(t+\theta)\mathrm d\theta\big)\mathrm dt
+x_2(t)\mathrm dB_2(t),\\
\end{array}
\right.
\end{align}
with the initial data $(\xi_1(\theta), \xi_2(\theta))^T=(\theta^2, \sin(-\theta+2))^T$ for any $\theta\in[-1/2,0]$.
Letting $p=2$ and applying the Young inequality, we  derive that for any $\phi(\cdot)=(\phi_1(\cdot),\phi_2(\cdot))^T\in \mathcal C$,
\begin{align*}
&2(\phi(0))^Tf(\phi)+|g(\phi)|^2\nn\\
\leq&-3(|\phi_1(0)|^2+|\phi_2(0)|^2)
-4(|\phi_1(0)|^4+|\phi_2(0)|^4)+\frac{1}{2}|\phi_1(0)|\int_{-1/4}^{0}|\phi_2(\theta)|\cdot4\mathrm d\theta\nn\\
&+|\phi_2(0)|\int_{-1/2}^{0}|\phi_1(\theta)|^3\cdot 2\mathrm d\theta\nn\\
\leq&-\frac{11}{4}|\phi(0)|^2-\frac{15}{2}(|\phi_1(0)|^4+|\phi_2(0)|^4)
+\frac{1}{4}\int_{-1/4}^{0}|\phi_2(\theta)|^2\cdot4\mathrm d\theta
+\frac{3}{4}\int_{-1/2}^{0}|\phi_1(\theta)|^4\cdot2\mathrm d\theta\nn\\
\leq&-\frac{11}{4}|\phi(0)|^2-\frac{15}{4}|\phi(0)|^4+\frac{1}{4}\int_{-1/4}^{0}|\phi(\theta)|^2\cdot4\mathrm d\theta
+\frac{3}{4}\int_{-1/2}^{0}|\phi^4(\theta)|\cdot2\mathrm d\theta,
\end{align*}
which implies that (A2$'$) holds with $$b_1=\frac{11}{4},~b_2=\frac{1}{4},~b_3=\frac{15}{4},~b_4=\frac{3}{4},$$
 $$\rho_3(\theta)=4\textbf 1_{[-1/4,0]}(\theta),~\rho_4(\theta)=2,~~~\forall~\theta\in[-\frac{1}{2},0].$$
  Choose $\nu=2$ in \eqref{o4.5}.
In view of Theorem \ref{th2},  the exact solution   of  SFDE \eqref{e2}  satisfies
  \begin{align*}
   &\limsup_{t\rightarrow \infty}\frac{1}{t}\log(\E|x(t)|^{2})\leq -2,\nn\\
   &\limsup_{t\rightarrow \infty}\frac{1}{t}\log|x(t)| \leq -1,~~~\hbox{a.s.}
   \end{align*}
  For each $l\geq 1$, a direct computation derives
\begin{align*}
&\sup_{\|\phi\|\vee\|\bar\phi\|\leq l}|f(\phi)-f(\bar\phi)|\leq (4+18l^2) (|\phi(0)-\bar\phi(0)|^2+2\int_{-1/2}^{0}|\phi(\theta)-\bar\phi(\theta)|^2\mathrm d\theta)^{1/2},\nn\\
&\sup_{\|\phi\|\vee\|\bar\phi\|\leq l}|g(\phi)-g(\bar\phi)|^2\leq|\phi(0)-\bar\phi(0)|^2
\end{align*}
for all $\phi,~\bar\phi\in \mathcal C$.
Recalling \eqref{o3.1} we take
$$\Gamma(l)=4+18l^2,~~~\forall~l\geq 1.$$
This implies $$\Gamma^{-1}(l)=(\frac{l}{18}-\frac{2}{9})^{1/2}, ~~~\forall ~l\geq 22.$$
Let $\lambda=0.001$. Then, define
$$\Lambda^{\triangle, \lambda}_{\Gamma}(x)= \Big(|x|\wedge (\frac{11}{9\tr^{0.001}}-\frac{2}{9})^{1/2}\Big) \frac{x}{|x|},$$
By virtue of Theorem \ref{th6}, for any $\nu_1\in(0,2)$, there exists a $\hat\tr\in(0,1]$ such that for any $\tr\in(0,\hat\tr]$, truncated EM solution \eqref{o3.8} satisfies
$$
   \limsup_{k\rightarrow\infty}\frac{1}{k\tr}\log{(\E|Y^{\tr}(t_k)|^{2})}\leq -\nu_1, ~~\hbox{and}~~
    \limsup_{k\rightarrow\infty}\frac{1}{k\tr}\log{| Y^{\tr}(t_k)|} \leq- \frac{\nu_1}{2}, ~\hbox{a.s.}$$
Figure \ref{sfdemean} depicts the sample mean of the truncated EM solution $Y^{\tr}(t)$. Figure \ref{sfdepp}
depicts the sample paths of the EM solution and the truncated EM solution.
\begin{figure}[htp]
  \begin{center}
\includegraphics[width=12cm,height=8cm]{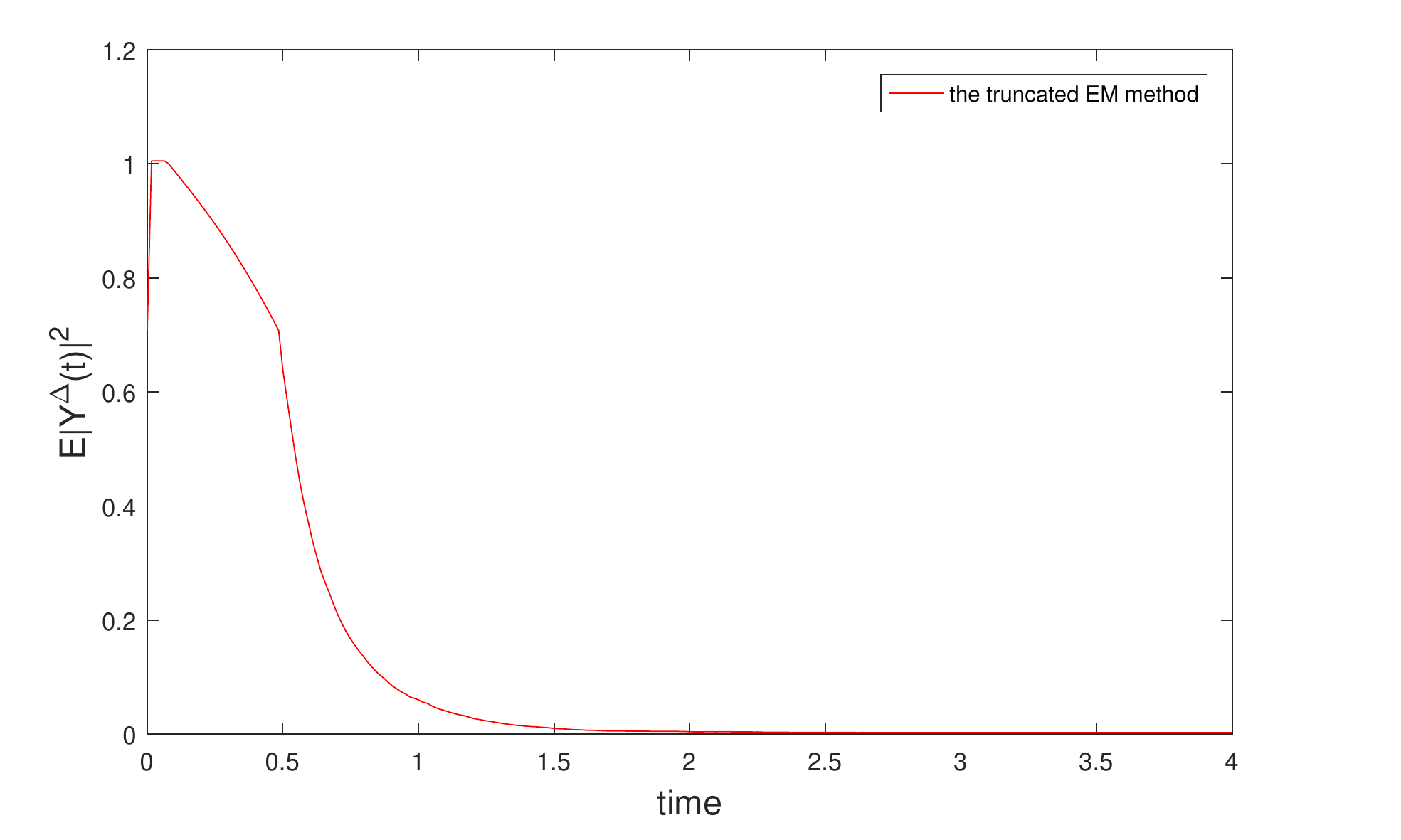}
   \end{center}
 \caption{The sample mean of $Y^{\tr}(t)$ for 1000 sample points as $\tr=2^{-6}$.  }
\label{sfdemean}
\end{figure}
\begin{figure}[htp]
  \begin{center}
\includegraphics[width=12cm,height=8cm]{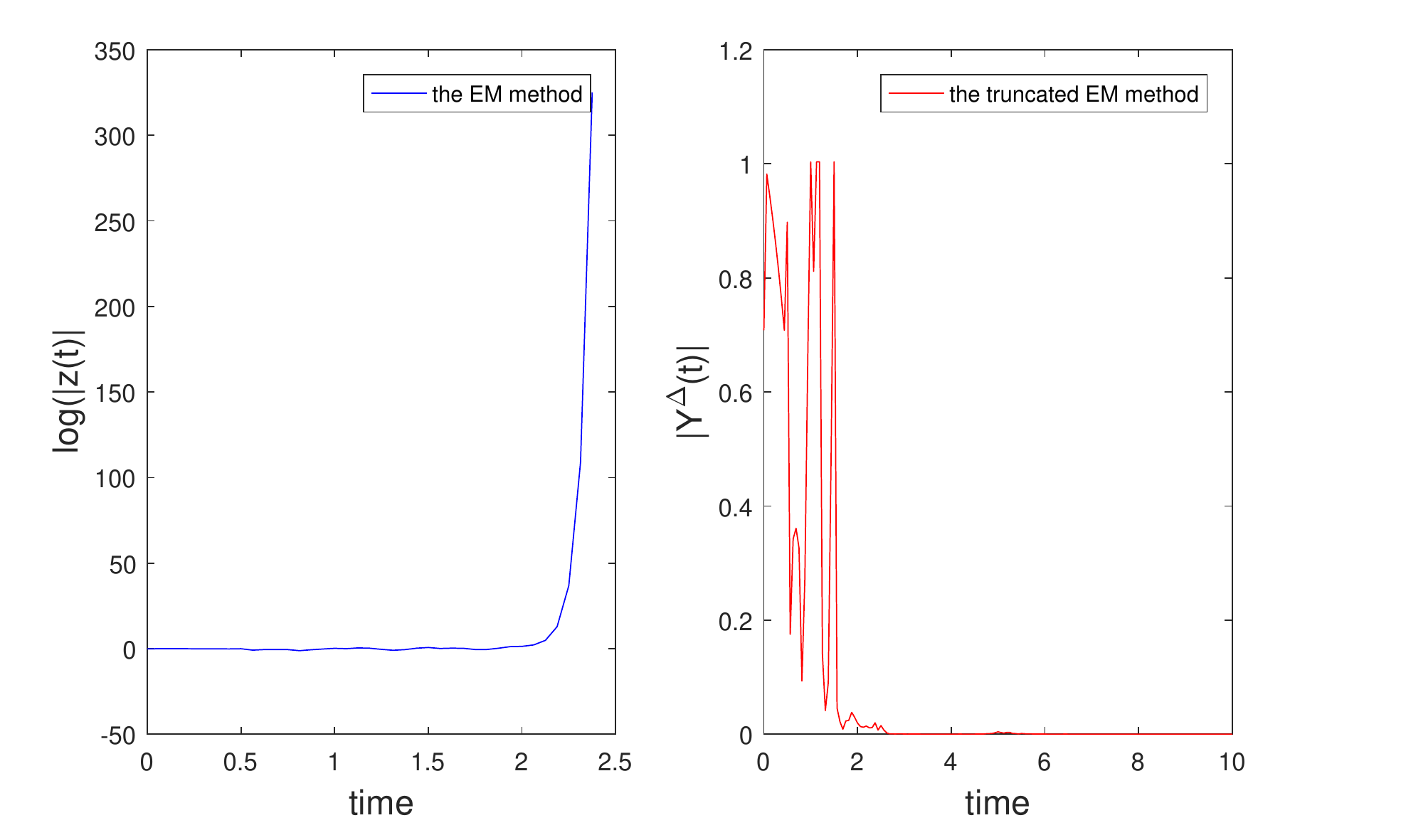}
   \end{center}
 \caption{ (Left) The sample path of EM solution $\ln(z(t))$ as $\tr=2^{-6}$. (Right) The sample path of truncated EM solution $Y^{\tr}(t)$ as $\tr=2^{-6}$.   }
\label{sfdepp}
\end{figure}

}
\end{expl}

\end{document}